\documentclass[a4paper]{article}

\usepackage[latin1]{inputenc} 
\usepackage{graphicx}
\usepackage[pdftex,bookmarks]{hyperref}
\usepackage{amsmath}    
\usepackage{amssymb}
\usepackage{amsmath,amsthm}
\usepackage{textcomp}
\usepackage[all]{xy}
\usepackage{geometry}

\newtheorem{teo}{Theorem}[section]

\newtheorem{cor}{Corollary}[section]
\newtheorem{prop}{Proposition}[section]

\DeclareMathOperator{\im}{im}
\DeclareMathOperator{\ind}{ind}
\DeclareMathOperator{\dvol}{dvol}

\DeclareMathOperator{\supp}{supp}

\DeclareMathOperator{\vol}{vol}
\DeclareMathOperator{\reg}{reg}
\DeclareMathOperator{\sing}{sing}
\DeclareMathOperator{\End}{End}

\DeclareMathOperator{\Hom}{Hom}

\DeclareMathOperator{\Tr}{Tr}

\DeclareMathOperator{\abs}{abs}
\DeclareMathOperator{\rel}{rel}

\title{\huge \bf On the Laplace-Beltrami operator on compact complex spaces}
\author{Francesco Bei  \bigskip \\
Dipartimento di Matematica, Universit\`a degli Studi di Padova,\\ E-mail addresses: \ bei@math.unipd.it \     francescobei27@gmail.com }
\date{}

\begin{document}

\maketitle

\begin{abstract}
\noindent Let $(X,h)$ be a compact and irreducible Hermitian complex space of complex dimension $v>1$. In this paper we show that the Friedrichs extension of  both the Laplace-Beltrami operator and the Hodge-Kodaira Laplacian acting on functions  has discrete spectrum. Moreover we provide some estimates for the growth of the corresponding eigenvalues and we use these estimates to deduce that the associated heat operators are trace-class. Finally we give various applications to the Hodge-Dolbeault operator and to the Hodge-Kodaira Laplacian in the setting of Hermitian complex spaces of complex dimension $2$.
\end{abstract}



\tableofcontents

\section*{Introduction}

Complex projective varieties endowed with the Fubini-Study metric, and more generally Hermitian complex spaces, provide an important class of singular spaces with a rich and beautiful interplay between analysis, geometry and topology. On the analytic side an active research field is given by the study of the $L^2$ properties of the natural differential operators arising in this setting, such as the Dolbeault operator, the de Rham differential  and the associated Dirac operators and Laplacians. Just to cite a small portion of the literature devoted to these questions we can mention here  \cite{BRLU}, \cite{BPS},   \cite{LT},  \cite{TOh}, \cite{OvRu}, \cite{OVV}, \cite{OV}, \cite{PS}, \cite{PaS},  \cite{Ruppe},  \cite{JRu}  and \cite{KIY}. The topics usually investigated in these papers  include the existence of self-adjoint extensions with discrete spectrum for the Hodge and the Hodge-Kodaira Laplacians, the $L^2$-$\overline{\partial}$-cohomology and $L^2$-de Rham cohomology, the $L^2$-Stokes theorem for the de Rham operator and so on. Let us now carry on by explaining  the main goal of this paper: consider a compact and irreducible Hermitian complex space $(X,h)$. We are interested in the properties of the Laplace-Beltrami operator and the Hodge-Kodaira Laplacian acting on functions 
\begin{equation}
\label{rinocefalo}
\Delta:C^{\infty}_c(\reg(X))\rightarrow C^{\infty}_c(\reg(X))
\end{equation}

\begin{equation}
\label{rinocefalos}
\Delta_{\overline{\partial}}:C^{\infty}_c(\reg(X))\rightarrow C^{\infty}_c(\reg(X))
\end{equation}
where both the operators above are viewed as unbounded, densely defined and closable operators acting on $L^2(\reg(X),h)$. Our main results show the existence of a self-adjoint extension with discrete spectrum for both \eqref{rinocefalo} and \eqref{rinocefalos} with an estimate for the growth of the corresponding eigenvalues. Let us go into some more details by describing how the paper is sort out. The first section contains some background materials. The second section is the main one and is devoted to the study of \eqref{rinocefalo} and \eqref{rinocefalos}. The first result we prove is a Sobolev embedding theorem in the setting of compact and irreducible Hermitian complex spaces: 
\begin{teo}
\label{euro}
Let $(X,h)$ be a compact and irreducible Hermitian complex space of complex dimension $v$. Then we have the following properties:
\begin{enumerate}
\item $W^{1,2}_0(\reg(X),h)=W^{1,2}(\reg(X),h)$.
\item Assume that $v>1$. Then we have a continuous inclusion $W^{1,2}(\reg(X),h)\hookrightarrow L^{\frac{2v}{v-1}}(\reg(X),h)$.
\item Assume that $v>1$. Then we have a compact inclusion  $W^{1,2}(\reg(X),h)\hookrightarrow L^2(\reg(X),h)$.
\end{enumerate}
\end{teo}

Then we proceed by using the above result to show the existence of a self-adjoint extension with discrete spectrum of \eqref{rinocefalo} providing in addition an asymptotic estimate for the growth of the corresponding  eigenvalues:
\begin{teo}
\label{Beltramide}
Let $(X,h)$ be a compact and irreducible Hermitian complex space of complex dimension $v>1$.  Consider the Laplace-Beltrami operator $\Delta:C^{\infty}_c(\reg(X))\rightarrow C^{\infty}_c(\reg(X))$ and let 
\begin{equation}
\label{fistw}
\Delta^{\mathcal{F}}:L^2(\reg(X),h)\rightarrow L^2(\reg(X),h)
\end{equation}
be its Friedrichs extension. Then \eqref{fistw} has discrete spectrum and 
\begin{equation}
\label{corew}
\{f\in C^{\infty}(\reg(X)): f\in L^2(\reg(X),h), df\in L^2\Omega^1(\reg(X),h), d^t(df)\in L^2(\reg(X),h)\}
\end{equation}
 is a core domain for \eqref{fistw}.\\Let $$0= \lambda_1<\lambda_2\leq...\leq \lambda_k\leq...$$ be the eigenvalues of \eqref{fistw}. Then 
\begin{equation}
\label{asyw}
\lim\inf \lambda_k k^{-\frac{1}{v}}>0
\end{equation}
as $k\rightarrow \infty$. Finally let 
\begin{equation}
\label{lovew}
e^{-t\Delta^{\mathcal{F}}}:L^2(\reg(X),h)\rightarrow L^2(\reg(X),h)
\end{equation}
be the heat operator associated to \eqref{fistw}. Then \eqref{lovew} is a trace class operator and its trace satisfies the following estimate
\begin{equation}
\label{cisiamow}
\Tr(e^{-t\Delta^{\mathcal{F}}})\leq C\vol_h(\reg(X))t^{-v}
\end{equation}
 for $0<t\leq 1$ and for some constant $C>0$.
\end{teo}

We observe that in the setting of complex projective varieties whose regular part is endowed with the Fubini-Study metric Th. \ref{euro} and Th. \ref{Beltramide} have been proved in \cite{LT}. Moreover Th. \ref{Beltramide} had been already proved in \cite{MasNag} and \cite{Pa} in the setting of projective varieties with isolated singularities and complex dimension $2$ and $3$ respectively.  Concerning \eqref{rinocefalos} we prove  similar results that we can summarize in the next two theorems:

\begin{teo}
\label{hklaplaciand}
Let $(X,h)$ be a compact and irreducible Hermitian complex space of complex dimension $v$. There exist positive constants $a$ and $b$ such that for each $f\in C^{\infty}_c(\reg(X))$ we have 
\begin{equation}
\label{apriled}
\|df\|^2_{L^2\Omega^1(\reg(X),h)}\leq a\|f\|^2_{L^2(\reg(X),h)}+b\|\overline{\partial}f\|^2_{L^2\Omega^{0,1}(\reg(X),h)}.
\end{equation}
Assume now that $v>1$.Consider the Hodge-Kodaira Laplacian acting on functions  $\Delta_{\overline{\partial}}:C^{\infty}_c(\reg(X))\rightarrow C^{\infty}_c(\reg(X))$ and let 
\begin{equation}
\label{wised}
\Delta^{\mathcal{F}}_{\overline{\partial}}:L^2(\reg(X),h)\rightarrow L^2(\reg(X),h)
\end{equation}
be its Friedrichs extension. We have the following properties:
\begin{enumerate}
\item Let $f\in L^{2}(\reg(X),h)$. Then $f\in \mathcal{D}(\overline{\partial}_{\min})$ if and only if $f\in W^{1,2}(\reg(X),h)$. Moreover the graph norm induced by $\overline{\partial}_{\min}$ and the norm of $W^{1,2}(\reg(X),h)$ are equivalent. Therefore, at the level of topological vector spaces, we have the following equality: $\mathcal{D}(\overline{\partial}_{\min})=W^{1,2}(\reg(X),h)$.
\item We have a continuous inclusion $\mathcal{D}(\overline{\partial}_{\min})\hookrightarrow L^{\frac{2v}{v-1}}(\reg(X),h)$ where $\mathcal{D}(\overline{\partial}_{\min})$ is endowed with the corresponding graph norm.
\item The natural inclusion $\mathcal{D}(\overline{\partial}_{\min})\hookrightarrow L^2(\reg(X),h)$ is a compact operator where $\mathcal{D}(\overline{\partial}_{\min})$ is endowed with the corresponding graph norm.
\item The operator $\Delta_{\overline{\partial}}^{\mathcal{F}}$ in \eqref{wised} has discrete spectrum.
\end{enumerate}
\end{teo}

\begin{teo}
\label{adissabeba}
In the setting of Th. \ref{hklaplaciand}. Let $(a,b)$ be any pair of positive numbers such that \eqref{apriled} holds true. Let $$0=\lambda_1<\lambda_2\leq\lambda_3\leq...\ \text{and}\ 0=\mu_1<\mu_2\leq \mu_3\leq...$$ be the eigenvalues of \eqref{fistw} and \eqref{wised} respectively. Then we have the following inequalities
\begin{equation}
\label{orociokd}
\mu_k\leq \lambda_k \leq a+b\mu_k
\end{equation}
for each $k\in \mathbb{N}$.
As a first consequence we have 
\begin{equation}
\label{asilod}
\lim\inf \mu_k k^{-\frac{1}{v}}>0
\end{equation}
as $k\rightarrow \infty$. Finally let 
\begin{equation}
\label{lottod}
e^{-t\Delta_{\overline{\partial}}^{\mathcal{F}}}:L^2(\reg(X),h)\rightarrow L^2(\reg(X),h)
\end{equation}
be the heat operator associated to \eqref{wised}. Then \eqref{lottod} is a trace class operator and its trace satisfies the following estimates
\begin{equation}
\label{matricianad}
e^{-at}\Tr(e^{-tb\Delta_{\overline{\partial}}^{\mathcal{F}}})\leq \Tr(e^{-t\Delta^{\mathcal{F}}})\leq \Tr(e^{-t\Delta_{\overline{\partial}}^{\mathcal{F}}})
\end{equation}
and, for each $0<t\leq 1$,
\begin{equation}
\label{agamennoned}
\Tr(e^{-t\Delta_{\overline{\partial}}^{\mathcal{F}}})\leq r\vol_h(\reg(X))t^{-v}
\end{equation}
where $r$ is some positive constant.
\end{teo}

We remark that if $\sing(X)$ is made of isolated points then the fact that \eqref{wised} has discrete spectrum is contained in \cite{Ruppe}. Furthermore we  stress  the fact that Th. \ref{euro}, Th. \ref{Beltramide}, Th. \ref{hklaplaciand} and Th. \ref{adissabeba} do not require  assumptions either on the dimension of $X$ or on $\sing(X)$. The remaining part of the second section contains some applications and corollaries of the above theorems. Finally  the third section is devoted to the Hodge-Kodaira Laplacian and the Hodge-Dolbeault operator in the setting of Hermitian complex spaces of complex dimension $2$. More precisely thanks to the above theorems we are able to extend to the case of $2$-dimensional Hermitian complex spaces the  results proved in \cite{FBei} in setting of complex projective surfaces. We conclude this introduction by recalling the following theorem from the third section. We refer to the first section for the definition of $\Delta_{\overline{\partial},2,q,\abs}$.
\begin{teo}
\label{lillottinaw}
Let $(X,h)$ be a compact and irreducible Hermitian complex  space of complex dimension $v=2$. For each $q=0,1,2$  we have the following properties:
\begin{enumerate}
\item $\Delta_{\overline{\partial},2,q,\abs}:L^2\Omega^{2,q}(\reg(X),h)\rightarrow L^2\Omega^{2,q}(\reg(X),h)$ has discrete spectrum.
\item $\overline{\partial}_{2,\max}+\overline{\partial}^t_{2,\min}:L^2\Omega^{2,\bullet}(\reg(X),h)\rightarrow L^2\Omega^{2,\bullet}(\reg(X),h)$ has discrete spectrum.
\item $\Delta_{\overline{\partial},2,q}^{\mathcal{F}}:L^2\Omega^{2,q}(\reg(X),h)\rightarrow L^2\Omega^{2,q}(\reg(X),h)$ has discrete spectrum.
\end{enumerate}
Let $q\in \{0,1,2\}$ and consider the operator
\begin{equation}
\label{resilew}
\Delta_{\overline{\partial},2,q,\abs}:L^2\Omega^{2,q}(\reg(X),h)\rightarrow L^2\Omega^{2,q}(\reg(X),h).
\end{equation}
Let $$0\leq \lambda_1\leq \lambda_2\leq...\leq \lambda_k\leq...$$ be the eigenvalues of \eqref{resilew}. Then we have the following asymptotic inequality
\begin{equation}
\label{resilesw}
\lim \inf \lambda_k k^{-\frac{1}{2}}>0
\end{equation}
as $k\rightarrow \infty$.\\ Consider now the heat operator associated to \eqref{resilew}
\begin{equation}
\label{resilexw}
e^{-t\Delta_{\overline{\partial},2,q,\abs}}:L^2\Omega^{2,q}(\reg(X),h)\rightarrow L^2\Omega^{2,q}(\reg(X),h).
\end{equation}
Then \eqref{resilexw} is a trace class operator and its trace satisfies the following estimate
\begin{equation}
\label{resilezw}
\Tr(e^{-t\Delta_{\overline{\partial},2,q,\abs}})\leq C_qt^{-2}
\end{equation}
for $t\in (0,1]$ and some constant $C_q>0$.
\end{teo}

\vspace{1 cm}

\noindent\textbf{Acknowledgments.}  This work was performed within the framework of the LABEX MILYON
(ANR-10-LABX-0070) of Universit\'e de Lyon, within the program "Investissements d'Avenir" (ANR-11-IDEX-0007) operated by the French National Research Agency (ANR).

\section{Background material}

We start with a brief reminder about $L^p$-spaces  and differential operators.  Let $(M,g)$  be an open and possibly incomplete Riemannian manifold of dimension $m$ and let $E$ be a vector bundle over $M$ of rank $k$ endowed with a metric $\rho$. Let $\dvol_g$ be the one-density associated to $g$.  
 For every $p$ with $1\leq p< \infty$  $L^{p}(M,E,g)$ is defined as the  space  of measurable sections $s$ such that    $$\|s\|_{L^{p}(M,E,g)}:=\left(\int_{M}|s|_{\rho}^p\dvol_g\right)^{1/p}<\infty$$ where $|s|_{\rho}:= (\rho(s,s))^{1/2}$. For each $p\in [1, \infty)$ we have a Banach space,  for each $p\in (1, \infty)$  we have a reflexive Banach space and  in the case $p=2$ we have a Hilbert space whose inner product is given by $$\langle s, t \rangle_{L^2(M,E,g)}:= \int_{M}\rho(s,t)\dvol_g.$$ Moreover $C^{\infty}_c(M,E)$,  the space of smooth sections with compact support,  is dense in $L^p(M,E,g)$ for $p\in [1,\infty).$ Finally $L^{\infty}(M,E,\rho)$ is defined as the space of measurable sections whose essential supp is bounded. 
Also in this case we get a Banach space. 
Consider  now  another vector bundle $F$ over $M$ endowed with a metric $\tau$. Let $P: C^{\infty}_c(M,E)\longrightarrow  C^{\infty}_c(M,F)$ be a differential operator of order $d$. The formal adjoint of $P$ is the differential operator $$P^t: C^{\infty}_c(M,F)\longrightarrow C^{\infty}_c(M,E)$$  uniquely characterized by the following property: for each $u\in C^{\infty}_c(M,E)$ and for each $v\in C^{\infty}_c(M,F)$ we have  $$\int_{M}\rho(u, P^tv)\dvol_g=\int_M\tau(Pu,v)\dvol_g.$$ From now on we restrict ourselves to the $L^2$ setting. We can look at $P$ as an unbounded, densely defined and closable operator  acting between $L^2(M,E,g)$ and $L^2(M,F,g)$. In general $P$ admits several different closed extensions. We recall now the definitions of the maximal and the minimal one.
The domain of the  maximal extension of $P:L^2(M,E,g)\longrightarrow L^2(M,F,g)$ is defined as $\mathcal{D}(P_{\max}):=\{s\in L^{2}(M,E,g): \text{there is}\ v\in L^2(M,F,g)\ \text{such that}$ $$\int_{M}\rho(s,P^t\phi)\dvol_g=\int_{M}\tau(v,\phi)\dvol_g$$ $\text{for each}\ \phi\in C^{\infty}_c(M,F,g)\}.$\ \text{In this case we put}\ $P_{\max}s=v.$  In other words the maximal extension of $P$ is the one defined in the distributional sense. The domain of the minimal extension of $P:L^2(M,E,g)\longrightarrow L^2(M,F,g)$ is defined as $\mathcal{D}(P_{\min}):=\{s\in L^{2}(M,E,g)\ \text{such that there is a sequence}\ \{s_i\}\in C_c^{\infty}(M,E)\ \text{with}$ $$s_i\rightarrow s\ \text{in}\ L^{2}(M,E,g)\ \text{and}\ Ps_i\rightarrow w\ \text{in}\ L^2(M,F,g)$$ $\text{to some }\ w\in L^2(M,F,g)\}.$ We put $P_{\min}s=w.$ Briefly the minimal extension of $P$ is the closure of $C^{\infty}_c(M,E)$ under the graph norm $\|s\|_{L^2(M,E,g)}+\|Ps\|_{L^2(M,F,g)}$.  It is immediate to check that 
\begin{equation}
\label{emendare}
P_{\max}^*=P^t_{\min}\ \text{and that}\  P_{\min}^*=P^t_{\max}
\end{equation}
 that is  $P^t_{\max/\min}:L^2(M,F,g)\rightarrow L^2(M,E,g)$ is the Hilbert space adjoint of $P_{\min/\max}$ respectively. 
We proceed now by recalling the notion of  Sobolev space associated to a metric connection. Consider again a vector bundle $E$ endowed with the metric $\rho$. Let $\nabla:C^{\infty}(M,E)\longrightarrow C^{\infty}(M,T^*M\otimes E)$ be a metric connection, that is a connection which satisfies the following property: for each $s, v\in C^{\infty}(M,E)$ we have $d(h(s,v))=\rho(\nabla s, v)+\rho(s, \nabla v)$. Clearly $\rho$ and $g$  induce in a natural way a metric on $T^*M\otimes E$ that we label by $\tilde{\rho}$. Let $\nabla^t:C^{\infty}_c(M,T^*M\otimes E)\longrightarrow C^{\infty}_c(M,E)$ be  the formal adjoint of $\nabla$ with respect to $\tilde{\rho}$ and $g$. Then  the Sobolev space $W^{1,2}(M,E)$ is  defined in the following way: $W^{1,2}(M,E):=\{s\in L^2(M,E): \text{there is}\ v\in L^2(M,T^*M\otimes E)\ \text{such that}$ $$\int_{M}\rho(s,\nabla^t\phi)\dvol_g=\int_{M}\tilde{\rho}(v,\phi)\dvol_g$$ $\text{for each}\ \phi\in C^{\infty}_c(M,T^*M\otimes E)\}$. Looking at $\nabla$ as a differential operator  we have $W^{1,2}(M,E)=\mathcal{D}(\nabla_{\max})$.
Moreover we have also   the Sobolev space $W^{1,2}_0(M,E)$ whose definition is the following: $W^{1,2}_0(M,E):=\{s\in L^2(M,E)\ \text{such that there is a sequence}\ \{s_i\}\in C_c^{\infty}(M,E)\ \text{with}$ $$s_i\rightarrow s\ \text{in}\ L^2(M,E)\ \text{and}\ \nabla s_i\rightarrow w\ \text{in}\ L^2(M,T^*M\otimes E)$$ $\text{to some }\ w\in L^2(M,T^*M\otimes E)\}.$
Similarly to previous case we have $W^{1,2}_0(M,E)=\mathcal{D}(\nabla_{\min})$. Clearly $W^{1,2}(M,E)$ and $W^{1,2}_0(M,E)$ are two Hilbert spaces. Concerning the notation, when $E$ is the trivial bundle $M\times \mathbb{R}$, we will write $W^{1,2}(M,g)$ and $W^{1,2}_0(M,g)$. In this case it is straightforward to check  that $W^{1,2}(M,g)=\mathcal{D}(d_{\max})$ and $W_0^{1,2}(M,g)=\mathcal{D}(d_{\min})$ where $d_{\max/\min}$ is respectively the maximal/minimal extension of the de Rham operator acting on functions $d:L^2(M,g)\rightarrow L^2\Omega^1(M,g)$.  We recall now the following results as they will be used later.
\begin{prop}
\label{locin}
Let $(M,g)$ be an open and possibly incomplete Riemannian manifold of dimension $m>2$. Let $U\subset M$ be an open subset with compact closure. Then we have the following continuous inclusion $$W^{1,2}_0(U,g|_{U})\hookrightarrow L^{\frac{2m}{m-2}}(U,g|_U).$$
\end{prop}
\begin{proof}
Since $\overline{U}$ is compact  there exists a compact submanifold with boundary $N\subset M$ such that $\overline{U}\subset A$ where with $A$ we label the interior of $N$. Therefore we have a continuous inclusion $W^{1,2}_0(U,g|_{U})\hookrightarrow W^{1,2}_0(A,g|_{A})$. Moreover, according to \cite{TAU} Th. 2.30,  we have a continuous inclusion $W^{1,2}_0(A,g|_{A})\hookrightarrow L^{\frac{2m}{m-2}}(A,g|_A).$ Therefore the statement of this proposition follows  composing the two continuous inclusions $$W^{1,2}_0(U,g|_{U})\hookrightarrow W^{1,2}_0(A,g|_{A})\hookrightarrow L^{\frac{2m}{m-2}}(A,g|_A)$$ and noticing that  $\|f\|_{L^{\frac{2m}{m-2}}(U,g|_U)}=\|f\|_{L^{\frac{2m}{m-2}}(A,g|_A)}$ for each $f\in W^{1,2}_0(U,g|_U)$.
\end{proof}

\begin{prop}
\label{dardo}
Let $(M,g)$ be an open and possibly incomplete Riemannian manifold of dimension $m$. Let $Y\subset M$ be a submanifold of $M$ and let $h$ be the Riemannian metric on $Y$ given by $h:=i_Y^*g$ where $i_Y:Y\rightarrow M$ is the canonical inclusion. Let $\omega\in \Omega^k(M)\cap L^{\infty}\Omega^k(M,g)$. Then $i_Y^*\omega\in L^{\infty}\Omega^{k}(Y,h)$ and $$\|i_Y^*\omega\|_{L^{\infty}\Omega^{k}(Y,h)}\leq \|\omega\|_{L^{\infty}\Omega^{k}(M,g)}.$$
\end{prop}
\begin{proof}
This follows decomposing $T^*M|_Y$ as $T^*M|_Y=T^*Y\oplus N^*_Y$, with $N^*_Y=(T^*_Y)^{\bot}$, and then performing an easy calculation of linear algebra. See for instance the remark in \cite{PS} page  611.
\end{proof}
%

In the remaining part of this introductory section we specialize in the case of complex manifolds. Our aim here is to  introduce   some notations and to recall  some results from the general theory of Hilbert complexes applied to the Dolbeault complex.  We invite the reader to consult  \cite{BL} for the proofs. Assume  that $(M,g)$ is a complex manifold of real dimension $2m$.  As usual with $\Lambda^{p,q}(M)$ we denote the bundle $\Lambda^p(T^{1,0}M)^*\otimes \Lambda^q(T^{0,1}M)^*$ and by $\Omega^{p,q}(M)$, $\Omega^{p,q}_c(M)$ we denote respectively the space of sections, sections with compact support,  of $\Lambda^{p,q}(M)$. On the bundle $\Lambda^{p,q}(M)$ we consider the Hermitian metric induced by $g$ and with a little abuse of notation we still label it by $g$. With $L^2\Omega^{p,q}(M,g)$ we denote the Hilbert space of $L^2$-$(p,q)$-forms. The Dolbeault operator acting on $(p,q)$-forms is labeled by $\overline{\partial}_{p,q}:\Omega^{p,q}(M)\rightarrow \Omega^{p,q+1}(M)$ and similarly  we have the operator $\partial_{p,q}:\Omega^{p,q}(M)\rightarrow \Omega^{p+1,q}(M)$. When we look at $\overline{\partial}_{p,q}:L^2\Omega^{p,q}(M,g)\rightarrow L^2\Omega^{p,q+1}(M,g)$ as an unbounded and densely defined operator with domain $\Omega_c^{p,q}(M)$ we label by $\overline{\partial}_{p,q,\max/\min}:L^2\Omega^{p,q}(M,g)\rightarrow L^2\Omega^{p,q+1}(M,g)$ respectively its maximal and minimal extension. Analogous meaning has the notation $\partial_{p,q,\max/\min}:L^2\Omega^{p,q}(M,g)\rightarrow L^2\Omega^{p+1,q}(M,g)$. In the case of functions we will simply write $\overline{\partial}:C^{\infty}(M)\rightarrow \Omega^{0,1}(M)$, $\overline{\partial}_{\max/\min}:L^2(M,g)\rightarrow L^2\Omega^{0,1}(M,g)$ and analogously $\partial:C^{\infty}(M)\rightarrow \Omega^{1,0}(M)$ and  $\partial_{\max/\min}:L^2(M,g)\rightarrow L^2\Omega^{1,0}(M,g)$. With $\overline{\partial}_{p,q}^t:\Omega^{p,q+1}_c(M)\rightarrow \Omega^{p,q}_c(M)$ and $\partial_{p,q}^t:\Omega^{p+1,q}_c(M)\rightarrow \Omega^{p,q}_c(M)$ we mean the formal adjoint of $\overline{\partial}_{p,q}:\Omega^{p,q}_c(M)\rightarrow \Omega^{p,q+1}_c(M)$ and $\partial_{p,q}:\Omega^{p,q}_c(M)\rightarrow \Omega^{p+1,q}_c(M)$ respectively. It is easy to check that if $\omega\in \mathcal{D}(\overline{\partial}_{p,q,\max/\min})$ then $\overline{\partial}_{p,q,\max/\min}\omega\in \ker(\overline{\partial}_{p,q+1,\max/\min})$. In this way we get two complexes, the minimal Dolbeault complex and maximal Dolbeault complex, which are in particular Hilbert complexes in the sense of Br\"uning-Lesch \cite{BL}. The maximal and the minimal $L^2$-$\overline{\partial}$-cohomology of $(M,g)$ is   respectively the cohomology  of the maximal and the minimal Dolbeault complex, that is:  
\begin{equation}
\label{chimar}
H^{p,q}_{2,\overline{\partial}_{\max}}(M,g):=\frac{\ker(\overline{\partial}_{p,q,\max})}{\im(\overline{\partial}_{p,q-1,\max})}\ \text{and}\ H^{p,q}_{2,\overline{\partial}_{\min}}(M,g):=\frac{\ker(\overline{\partial}_{p,q,\min})}{\im(\overline{\partial}_{p,q-1,\min})}.
\end{equation}
It is not difficult to see that if $H^{p,q}_{2,\overline{\partial}_{\max}}(M,g)$ is finite dimensional then $\im(\overline{\partial}_{p,q-1,\max})$ is closed and analogously if $H^{p,q}_{2,\overline{\partial}_{\min}}(M,g)$ is finite dimensional then $\im(\overline{\partial}_{p,q-1,\min})$ is closed. Also in this case we remark that analogous definitions and constructions hold for the complexes associated to $\partial_{p,q,\max/\min}$. Consider now the Hodge-Kodaira Laplacian  $$\Delta_{\overline{\partial},p,q}:\Omega^{p,q}_c(M)\rightarrow \Omega^{p,q}_c(M),\ \Delta_{\overline{\partial},p,q}:=\overline{\partial}_{p,q-1}\circ\overline{\partial}^t_{p,q-1}+\overline{\partial}_{p,q}^t\circ \overline{\partial}_{p,q}.$$
In the case of functions, that is $(p,q)=(0,0)$, we will simply write $\Delta_{\overline{\partial}}:C^{\infty}_c(M)\rightarrow C^{\infty}_c(M)$.  We recall the definition of  two important self-adjoint extensions of $\Delta_{\overline{\partial},p,q}$:
\begin{equation}
\label{asdf}
\overline{\partial}_{p,q-1,\max}\circ \overline{\partial}_{p,q-1,\min}^t+\overline{\partial}_{p,q,\min}^t\circ \overline{\partial}_{p,q,\max}:L^2\Omega^{p,q}(M,g)\rightarrow L^2\Omega^{p,q}(M,g)
\end{equation} 
and 
\begin{equation}
\label{buio}
\overline{\partial}_{p,q-1,\min}\circ \overline{\partial}_{p,q-1,\max}^t+\overline{\partial}_{p,q,\max}^t\circ \overline{\partial}_{p,q,\min}:L^2\Omega^{p,q}(M,g)\rightarrow L^2\Omega^{p,q}(M,g)
\end{equation}
called respectively the absolute and the relative extension. The operator \eqref{asdf}, the absolute extension, is labeled in general with  $\Delta_{\overline{\partial},p,q,\abs}$ and its domain is given by $$\mathcal{D}(\Delta_{\overline{\partial},p,q,\abs})=\left\{\omega\in \mathcal{D}(\overline{\partial}_{p,q,\max})\cap \mathcal{D}(\overline{\partial}_{p,q-1,\min}^t):\overline{\partial}_{p,q,\max}\omega \in \mathcal{D}(\overline{\partial}^t_{p,q,\min})\ \text{and}\ \overline{\partial}_{p,q-1,\min}^t\omega \in \mathcal{D}(\overline{\partial}_{p,q-1,\max})\right\}.$$
The operator \eqref{buio}, the relative extension,  is labeled in general with  $\Delta_{\overline{\partial},p,q,\rel}$ and its domain is given by $$\mathcal{D}(\Delta_{\overline{\partial},p,q,\rel})=\left\{\omega\in \mathcal{D}(\overline{\partial}_{p,q,\min})\cap \mathcal{D}(\overline{\partial}_{p,q-1,\max}^t):\overline{\partial}_{p,q,\min}\omega \in \mathcal{D}(\overline{\partial}^t_{p,q,\max})\ \text{and}\ \overline{\partial}_{p,q-1,\max}^t\omega \in \mathcal{D}(\overline{\partial}_{p,q-1,\min})\right\}.$$
Briefly $\Delta_{\overline{\partial},p,q,\abs}$ are the Laplacians associated to the maximal Dolbeault complex while $\Delta_{\overline{\partial},p,q,\rel}$ are the Laplacians associated to the minimal Dolbeault complex.
Consider now the Hodge-Dolbeault operator $\overline{\partial}_{p}+\overline{\partial}^t_{p}:\Omega_c^{p,\bullet}(M)\rightarrow \Omega^{p,\bullet}_c(M)$ where  $\Omega^{p,\bullet}_c(M)$ stands for $\bigoplus_{q=0}^m\Omega_c^{p,q}(M)$. We can define two self-adjoint extensions  of $\overline{\partial}_p+\overline{\partial}^t_p$ taking  
\begin{equation}
\label{twoself}
\overline{\partial}_{p,\max}+\overline{\partial}^t_{p,\min}:L^2\Omega^{p,\bullet}(M,g)\rightarrow L^2\Omega^{p,\bullet}(M,g)
\end{equation}
\begin{equation} 
\label{twoselfs}
\overline{\partial}_{p,\min}+\overline{\partial}^t_{p,\max}:L^2\Omega^{p,\bullet}(M,g)\rightarrow L^2\Omega^{p,\bullet}(M,g)
\end{equation}
where clearly $L^2\Omega^{p,\bullet}(M,g)=\bigoplus_{q=0}^mL^2\Omega^{p,q}(M,g)$. The domain of $\overline{\partial}_{p,\max}+\overline{\partial}^t_{p,\min}$ is given by $\mathcal{D}(\overline{\partial}_{p,\max})\cap\mathcal{D}(\overline{\partial}^t_{p,\min})$ where $\mathcal{D}(\overline{\partial}_{p,\max})=\bigoplus_{q=0}^m\mathcal{D}(\overline{\partial}_{p,q,\max})$ and $\mathcal{D}(\overline{\partial}^t_{p,\min})=\bigoplus_{q=0}^m\mathcal{D}(\overline{\partial}^t_{p,q,\min})$. Analogously the domain of $\overline{\partial}_{p,\min}+\overline{\partial}^t_{p,\max}$ is given by  $\mathcal{D}(\overline{\partial}_{p,\min})\cap\mathcal{D}(\overline{\partial}^t_{p,\max})$ where $\mathcal{D}(\overline{\partial}_{p,\min})=\bigoplus_{q=0}^m\mathcal{D}(\overline{\partial}_{p,q,\min})$ and $\mathcal{D}(\overline{\partial}^t_{p,\max})=\bigoplus_{q=0}^m\mathcal{D}(\overline{\partial}^t_{p,q,\max})$. In this case we can look at $\overline{\partial}_{p,\min/\max}+\overline{\partial}^t_{p,\max/\min}$  as the Dirac operator associated  to the minimal/maximal Dolbeault complex respectively. 
Clearly, if we replace the operator $\overline{\partial}_{p,q}$ with $\partial_{p,q}$, then we get the analogous definitions and properties for the operators $\partial_{p,q,\max/\min}$, $\Delta_{\partial,p,q}$, $\Delta_{\partial,p,q,\abs/\rel}$, $\partial_{q}+\partial^t_{q}$ and $\partial_{q,\max/\min}+\partial^t_{q,\min/\max}$. 
Moreover, according to \cite{Huy} page 116, we have 
\begin{equation}
\label{cicci}
\partial^t_{p,q}=-*\overline{\partial}_{m-q,m-p-1}*\ \text{and}\ \overline{\partial}^t_{p,q}=-*\partial_{m-q-1,m-p}*
\end{equation}
and from these relations we easily get that 
\begin{equation}
\label{cicuta}
\partial^t_{p,q,\max/\min}=-*\overline{\partial}_{m-q,m-p-1,\max/\min}*\ \text{and}\ \overline{\partial}^t_{p,q,\max/\min}=-*\partial_{m-q-1,m-p,\max/\min}*
\end{equation}
where $*:L^2\Omega^{p,q}(M,g)\rightarrow L^2\Omega^{m-q,m-p}(M,g)$ is the unitary operator induced by the  Hodge star operator. Now let us label by 
\begin{equation}
\label{conj}
c: T^{1,0}M\rightarrow T^{0,1}M
\end{equation}
the $\mathbb{C}$-antilinear map given by  complex conjugation. With a little abuse of notation we label again by
\begin{equation}
\label{conju}
c:\Lambda^{p,q}(M)\rightarrow \Lambda^{q,p}(M)
\end{equation}
 the natural map induced by \eqref{conj} and, still with a little abuse of notation, we label by $c$ the induced map on $(p,q)$-forms, that is
\begin{equation}
\label{conjform}
 c:\Omega^{p,q}(M)\rightarrow \Omega^{q,p}(M).
\end{equation} 
Clearly both \eqref{conju} and \eqref{conjform} are  $\mathbb{C}$-antilinear isomorphisms whose corresponding inverse is given by 
$c:\Lambda^{q,p}(M)\rightarrow \Lambda^{p,q}(M)$ and  $c:\Omega^{q,p}(M)\rightarrow \Omega^{p,q}(M)$ respectively.
Moreover \eqref{conjform} induces an isomorphism 
\begin{equation}
\label{conjformcs}
 c|_{\Omega^{p,q}_c(M)}:\Omega_c^{p,q}(M)\rightarrow \Omega^{q,p}_c(M).
\end{equation} 
 We have the following well known properties:

\begin{prop}
\label{comconj}
In the setting described above. On $\Omega^{p,q}(M)$  the following properties hold true: 
\begin{enumerate}
\item $\overline{\partial}_{p,q}=c\circ \partial_{q,p}\circ c.$
\item $\overline{\partial}_{p,q}^t=c\circ \partial_{q,p}^t\circ c$
\end{enumerate}
\end{prop}
\begin{proof}
For the first point see for instance \cite{Wells} Prop. 3.6. The second point follows using the first point, \eqref{cicci} and the fact that the Hodge star operator commutes with the complex conjugation.
\end{proof}

Now consider again $M$ endowed with a Hermitian metric $g$.  For each $\omega\in \Omega_c^{p,q}(M)$ it is easy to check that 
$g(\omega,\omega)=g(c(\omega), c(\omega))$. Using this  equality and the other properties recalled above, we easily get that $c$ induces a $\mathbb{C}$-antilinear operator 
\begin{equation}
\label{conjhil}
c:L^2\Omega^{p,q}(M,g)\rightarrow L^2\Omega^{q,p}(M,g)
\end{equation} 
which is bijective, continuous, with continuous inverse given by $c:L^2\Omega^{q,p}(M,g)\rightarrow L^2\Omega^{p,q}(M,g)$ and such that $\|\eta\|_{L^2\Omega^{p,q}(M,g)}=\|c(\eta)\|_{L^2\Omega^{q,p}(M,g)}$ for each $\eta\in L^2\Omega^{p,q}(M,g)$. 
Finally we conclude this introductory section with the following proposition.
\begin{prop}
\label{occhiodibue}
In the setting described above. The following properties hold true:
\begin{enumerate}
\item $c\left(\mathcal{D}(\overline{\partial}_{p,q,\max})\right)=\mathcal{D}(\partial_{q,p\max})$ and\  $\overline{\partial}_{p,q,\max}=c\circ \partial_{q,p,\max}\circ c.$
\item  $c\left(\mathcal{D}(\overline{\partial}_{p,q,\min})\right)=\mathcal{D}(\partial_{q,p\min})$ and\  $\overline{\partial}_{p,q,\min}=c\circ \partial_{q,p,\min}\circ c.$
\item $c\left(\mathcal{D}(\overline{\partial}^t_{p,q,\max})\right)=\mathcal{D}(\partial_{q,p\max}^t)$ and\  $\overline{\partial}_{p,q,\max}^t=c\circ \partial_{q,p,\max}^t\circ c.$
\item $c\left(\mathcal{D}(\overline{\partial}^t_{p,q,\min})\right)=\mathcal{D}(\partial_{q,p\min}^t)$ and\  $\overline{\partial}_{p,q,\min}^t=c\circ \partial_{q,p,\min}^t\circ c.$
\item $*\left(\mathcal{D}(\Delta_{\overline{\partial},p,q,\abs})\right)=\mathcal{D}(\Delta_{\partial,m-q,m-p,\rel})$ and 
$*\circ  \Delta_{\overline{\partial},p,q,\abs}=\Delta_{\partial,m-q,m-p,\rel}\circ *.$
\item $*\left(c(\mathcal{D}(\Delta_{\overline{\partial},p,q,\abs}))\right)=\mathcal{D}(\Delta_{\overline{\partial},m-p,m-q,\rel})$ and 
$*\circ c\circ \Delta_{\overline{\partial},p,q,\abs}=\Delta_{\overline{\partial},m-p,m-q,\rel}\circ *\circ c.$
\end{enumerate}
\end{prop}
\begin{proof}
This follows immediately by \eqref{cicuta}, Prop. \ref{comconj} and the properties of \eqref{conjhil}.
\end{proof}

\section{Main results}
This section contains the main results of this paper. We address the question of the existence of self-adjoint extensions with discrete spectrum for both the Laplace-Beltrami operator and the Hodge-Kodaira Laplacian acting on functions in the setting of compact and irreducible Hermitian complex spaces. We will show that in both cases the  Friedrichs extension has discrete spectrum. Moreover we will provide an estimate for the growth of the corresponding eigenvalues. For the definition and the main properties of the Friedrichs extension we refer to \cite{MaMa}. Concerning the topic of complex  spaces  we refer to \cite{GeFi} and \cite{GrRe}. Here we recall that an irreducible complex space $X$ is a reduced complex space such that $\reg(X)$, the regular part of $X$, is connected. Furthermore we recall  that a paracompact  and reduced complex space $X$ is said \emph{Hermitian} if  the regular part of $X$ carries a Hermitian metric $h$  such that for every point $p\in X$ there exists an open neighborhood $U\ni p$ in $X$, a proper holomorphic embedding of $U$ into a polydisc $\phi: U \rightarrow \mathbb{D}^N\subset \mathbb{C}^N$ and a Hermitian metric $g$ on $\mathbb{D}^N$ such that $(\phi|_{\reg(U)})^*g=h$, see for instance \cite{Ohsa} or \cite{JRu}. In this case we will write $(X,h)$ and with a little abuse of language we will say that $h$ is a Hermitian metric on $X$. Clearly any analytic sub-variety of a complex Hermitian manifold endowed with the metric induced by the restriction of the  metric of the ambient space is a Hermitian complex space. In particular, within this class of examples, we have any complex projective variety $V\subset \mathbb{C}\mathbb{P}^n$ endowed  with the K\"ahler metric induced by the Fubini-Study metric of $\mathbb{C}\mathbb{P}^n$. We are in position to state the first result of this section. It is a Sobolev embedding theorem in the setting of compact and irreducible Hermitian complex space.
\begin{teo}
\label{dollar}
Let $(X,h)$ be a compact and irreducible Hermitian complex space of complex dimension $v$. Then we have the following properties:
\begin{enumerate}
\item $W^{1,2}_0(\reg(X),h)=W^{1,2}(\reg(X),h)$.
\item Assume that $v>1$. Then we have a continuous inclusion $W^{1,2}(\reg(X),h)\hookrightarrow L^{\frac{2v}{v-1}}(\reg(X),h)$.
\item Assume that $v>1$. Then we have a compact inclusion  $W^{1,2}(\reg(X),h)\hookrightarrow L^2(\reg(X),h)$.
\end{enumerate}
\end{teo}

\begin{proof}
According to \cite{BeGu} or \cite{JRU} we know that $(\reg(X),h)$  is  parabolic. Now the \emph{first property} follows by Prop. 4.1 in \cite{FraBei}. 
The proof of the \emph{second property} is divided in two steps. First we prove that a Sobolev embedding exists locally. Then, using a suitable partition of unity, we patch together these local Sobolev embeddings in order to get a global one. So the first goal is to prove the following statement:
\begin{itemize}
\item For every $p\in X$ there exists an open neighborhood $V$ of $p$ such that we have a continuous inclusion $W_0^{1,2}(\reg(V),h|_{\reg(V)})\hookrightarrow L^{\frac{2v}{v-1}}(\reg(V),h|_{\reg(V)})$.
\end{itemize}
If $p\in \reg(X)$ it is clear that the above statement holds true. Therefore we are left to deal with the case $p\in \sing(X)$.
Since $(X,h)$ is a Hermitian complex space we know that  there exists an open neighborhood $U\ni p$ in $X$, a proper holomorphic embedding of $U$ into a polydisc $\phi: U \rightarrow \mathbb{D}^N\subset \mathbb{C}^N$ and a Hermitian metric $g$ on $\mathbb{D}^N$ such that $(\phi|_{\reg(U)})^*g=h|_{\reg(U)}$. Consider an open subset $V$ of $X$ with compact closure such that $p\in V$ and $\overline{V}\subset U$ and let $Y:=\phi(V)$. Let $g'=dz_1^2+...+dz_N^2$ be the standard K\"ahler metric on $\mathbb{C}^N$. Since $\overline{\phi(V)}\subset \mathbb{D}^N$ we can find an open subset $A$ of $\mathbb{C}^N$ with compact closure such that $\overline{Y}\subset A\subset \overline{A}\subset \mathbb{D}^N$. Hence $g|_{A}$ and $g'|_{A}$ are quasi isometric and this in turn implies that $\zeta:= (i|_{\reg(Y)})^*g$  is quasi isometric to $\zeta':=(i|_{\reg(Y)})^*g'$  where $i:\reg(Y)\hookrightarrow \mathbb{D}^N$ is the canonical inclusion. Consider now the K\"ahler manifold $(\reg(Y), \zeta')$. By the fact that $(\reg(Y), \zeta')$ is a K\"ahler submanifold of $(\mathbb{C}^N,g')$ we know that $(\reg(Y), \zeta')$ is a minimal submanifold of $(\mathbb{C}^N,g')$, that is its mean curvature vector field vanishes identically, see \cite{Gray} Lemma 6.26. Therefore the Sobolev embedding theorem of Michael-Simon, see \cite{MISI}, tells us that we have a continuous inclusion 
\begin{equation}
\label{brunello}
W_0^{1,2}(\reg(Y),\zeta')\hookrightarrow L^{\frac{2v}{v-1}}(\reg(Y),\zeta').
\end{equation}
Since $\zeta$ is quasi-isometric to $\zeta'$ we have a continuous inclusion  $$W_0^{1,2}(\reg(Y),\zeta)\hookrightarrow L^{\frac{2v}{v-1}}(\reg(Y),\zeta)$$ and finally, by the fact that $\phi(V)=Y$ and $(\phi|_{\reg(V)})^*g=h|_{\reg(V)}$ we get our desired continuous inclusion  
\begin{equation}
\label{symp}
W_0^{1,2}(\reg(V),h|_{\reg(V)})\hookrightarrow L^{\frac{2v}{v-1}}(\reg(V),h|_{\reg(V)}).
\end{equation}
This establishes the first step of the proof. Now we start with the second part of the proof. As previously said now the goal is to patch together these local Sobolev embeddings. In order to achieve this purpose we are going to build a suitable partition of unity. For every point $p\in \sing(X)$ let $U_p$ be an open neighborhood of $p$ with compact closure such that there exists a proper holomorphic embedding of $U_p$ into a polydisc $\phi_p: U_p \rightarrow \mathbb{D}^N\subset \mathbb{C}^N$ as above. Let $V_p$ be another open neighborhood of $p$ such that $\overline{V_p}\subset U_p$. As we have seen above we have a continuous inclusion $W^{1,2}_0(\reg(V_p),h|_{\reg(V_p)})\hookrightarrow L^{\frac{2v}{v-1}}(\reg(V_p),h|_{\reg(V_p)})$. In the same way, for every point $q\in \reg(X)$, consider an open neighborhood $V_q$ with compact closure such that $\overline{V_q}\subset \reg(X)$. According to Prop. \ref{locin} we have a continuous inclusion $W_0^{1,2}(V_q,h|_{V_q})\hookrightarrow L^{\frac{2v}{v-1}}(V_q,h|_{V_q})$. Hence, since $X$ is compact, we can find a finite number of points $q_1,...,q_n\in \reg(X)$ and $p_1,...,p_m\in \sing(X)$ such that $$\mathcal{W}:= \{V_{q_1},...,V_{q_n},V_{p_1},...,V_{p_m}\}$$ is a finite open cover of $X$, $\overline{V_{q_i}}\cap \sing(X)=\emptyset$ for each $i=1,...,n$ and the  Sobolev embeddings
\begin{equation}
\label{metz}
W^{1,2}_0(\reg(V_{p_j}),h|_{\reg(V_{p_j})})\hookrightarrow L^{\frac{2v}{v-1}}(\reg(V_{p_j}),h|_{\reg(V_{p_j})})\\
\quad\quad\quad  W_0^{1,2}(V_{q_i},h|_{V_{q_i}})\hookrightarrow L^{\frac{2v}{v-1}}(V_{q_i},h|_{V_{q_i}})
\end{equation}
hold for each $j=1,...,m$ and $i=1,...,n$. By the fact that $X$ is paracompact and Hausdorff we can find a continuous partition of unity $\{\alpha_{q_1},...,\alpha_{q_n},\beta_{p_1},...,\beta_{p_m}\}$ subordinated to the open cover $\mathcal{W}$. For any $q_i$ with $i\in\{1,...,n\}$ let $\alpha'_{q_i}:\reg(X)\rightarrow \mathbb{R}$ be a smooth function such that $\alpha'_{q_i}\geq \alpha_{q_i}$, $\supp(\alpha'_{q_i})\subset V_{q_i}$ and $\alpha'_{q_i}(p)=1$ if $\alpha_{q_i}(p)=1$.  Clearly such a function exists. Likewise for any $p_j$ with $j\in \{1,...,m\}$ consider a continuous functions $\beta'_{p_j}:X\rightarrow \mathbb{R}$ such that   $\beta'_{p_j}\geq \beta_{p_j}$, $\supp(\beta'_{p_j})\subset V_{p_j}$, $\beta'_{p_j}(p)=1$ if $\beta_{p_j}(p)=1$, and $\beta'_{p_j}|_{\reg(X)}$ is smooth. In order to show that also such a function exists we argue as follows. Let $Z_{p_j}:=\phi_{p_j}(U_{p_j})$ and let $Y_{p_j}:=\phi_{p_j}(V_{p_j})$. Let us define $\chi_{p_j}:=\beta_{p_j}\circ \phi_{p_j}^{-1}$. Then $\chi_{p_j}:Z_{p_j}\rightarrow \mathbb{R}$ is a continuous function whose support is contained in $Y_{p_j}=\phi_{p_j}(V_{p_j})$. Let $A_{p_j}\subset \mathbb{D}^N$ be an open subset with compact closure such that $\overline{A_{p_j}}\subset \mathbb{D}^N$ and $A_{p_j}\cap Z_{p_j}=Y_{p_j}$.  Let $\upsilon_{p_j}:\mathbb{D}^{N}\rightarrow \mathbb{R}$ be a continuous function  such that $\supp(\upsilon_{p_j})\subset A_{p_j}$ and  $\upsilon_{p_j}|_{Y_{p_j}}=\chi_{p_j}$. Consider now a smooth function \begin{equation}
\label{tolosa}
\tilde{\upsilon}_{p_j}:\mathbb{D}^{N}\rightarrow \mathbb{R}
\end{equation}
 such that $\supp(\tilde{\upsilon}_{p_j})\subset A_{p_j}$, $\tilde{\upsilon}_{p_j}\geq \upsilon_{p_j}$ and $\tilde{\upsilon}_{p_j}(x)=1$ if $\upsilon_{p_j}(x)=1$. Finally we define $\beta'_{p_j}:=\tilde{\upsilon}_{p_j}\circ \phi_{p_j}$. It is straightforward to check that  $\beta'_{p_j}$ satisfies the required properties. As a next step, for each $i\in \{1,...,n\}$, let us define $$\gamma_{q_i}:= \frac{\alpha'_{q_i}}{\sum_{i=1}^n \alpha'_{q_i}+\sum_{j=1}^m \beta'_{p_j}}.$$ Similarly for each $j\in \{1,...,m\}$ let us define $$\gamma_{p_j}:= \frac{\beta'_{p_j}}{\sum_{i=1}^n \alpha'_{q_i}+\sum_{j=1}^m \beta'_{p_j}}.$$ We claim that  
\begin{equation}
\label{goodpart}
\{\gamma_{q_1},...,\gamma_{q_n}, \gamma_{p_1},...,\gamma_{p_m}\}
\end{equation}
 is a continuous partition of unity subordinated to the open cover $\mathcal{W}$ such that: 
\begin{enumerate}
\item for each $i\in \{1,...,n\}$ $\gamma_{q_i}\in C^{\infty}_c(V_{q_i})\subset C^{\infty}_c(\reg(X))$,
\item for each $j\in \{1,...,m\}$ $\gamma_{p_j}\in  C_c(V_{p_j})\subset C_c(X)$ and $\gamma_{p_j}|_{\reg(X)}\in C^{\infty}(\reg(X))$,
\item there exists  a constant $c>0$ such that  
\begin{equation}
\label{mcvf}
\|d\gamma_{q_i}\|_{L^{\infty}\Omega^1(\reg(X),h)}\leq c\ \text{and}\ \|d(\gamma_{p_j}|_{\reg(X)})\|_{L^{\infty}\Omega^1(\reg(X),h)}\leq c
\end{equation}
for each $i\in \{1,...,n\}$ and for each $j\in \{1,...,m\}$ respectively.
\end{enumerate}
That $\{\gamma_{q_1},...,\gamma_{q_n}, \gamma_{p_1},...,\gamma_{p_m}\}$ is a continuous partition of unity subordinated to the open cover $\mathcal{W}$  follows immediately by the definitions of $\gamma_{q_i}$ and $\gamma_{p_j}$. For the same reason  the first two points in the above list are clear. We are therefore left to deal with \eqref{mcvf}. It is clear that for each $i\in \{1,...,n\}$ there exists a positive constant $c_i$ such that $\|d\gamma_{q_i}\|_{L^{\infty}\Omega^1(\reg(X),h)}\leq c_i$. Consider now any function $\gamma_{p_j}$ with $j\in \{1,...,m\}$ arbitrarily fixed. For simplicity let us label $\xi:=\sum_{i=1}^n \alpha'_{q_i}+\sum_{j=1}^m \beta'_{p_j}$. We have $$d(\gamma_{p_j}|_{\reg(X)})=\frac{\xi d(\beta'_{p_j}|_{\reg(X)})-\beta'_{p_j}d(\xi|_{\reg(X)})}{\xi^2}.$$ Clearly we have $\xi|_{\reg(X)}\in L^{\infty}(\reg(X),h)$, $\xi^{-2}|_{\reg(X)}\in L^{\infty}(\reg(X),h)$ and $\beta'_{p_j}|_{\reg(X)}\in L^{\infty}(\reg(X),h)$. Concerning $d(\beta'_{p_j}|_{\reg(X)})$ we have $d(\beta'_{p_j}|_{\reg(X)})\in L^{\infty}\Omega^1(\reg(X),h)$ if and only if $$d(\tilde{\upsilon}_{p_j}|_{\reg(Y_{p_j})})\in L^{\infty}\Omega^1(\reg(Y_{p_j}), \zeta_{p_j})$$ where we recall that  $\zeta_{p_j}:=(i|_{\reg(Y_{p_j})})^*g$, $i:\reg(Y_{p_{j}})\hookrightarrow \mathbb{D}^N$ is the canonical inclusion and $\tilde{\upsilon}_{p_j}$ is defined in \eqref{tolosa}.
Since $\tilde{\upsilon}_{p_j}\in C^{\infty}_{c}(A_{p_j})$ we have $d\tilde{\upsilon}_{p_j}\in L^{\infty}\Omega^1(A_{p_j},g|_{A_{p_j}})$ and this in turn implies immediately that $d(\tilde{\upsilon}_{p_j}|_{\reg(Y_{p_j})})\in$ $L^{\infty}\Omega^1(\reg(Y_{p_j}), \zeta_{p_j})$, see Prop. \ref{dardo}. Therefore we can conclude that $d(\beta'_{p_j}|_{\reg(X)})\in L^{\infty}\Omega^1(\reg(X),h)$. Concerning $d(\xi|_{\reg(X)})$ we have 
\begin{align}
\nonumber &\|d(\xi|_{\reg(X)})\|_{L^{\infty}\Omega^{1}(\reg(X),h)}=\|d(\sum_{i=1}^n \alpha'_{q_i}+\sum_{j=1}^m \beta'_{p_j}|_{\reg(X)})\|_{L^{\infty}\Omega^1\reg(X,h)}\leq\\
\nonumber &\sum_{i=1}^n\|d\alpha'_{q_i}\|_{L^{\infty}\Omega^1(\reg(X),h)}+\sum_{j=1}^m\|d(\beta'_{p_j}|_{\reg(X)})\|_{L^{\infty}\Omega^1(\reg(X),h)}<\infty
\end{align}
because, as we have previously seen, $d\alpha'_{q_i}\in L^{\infty}\Omega^{1}(\reg(X),h)$ and $d(\beta'_{p_j}|_{\reg(X)})\in L^{\infty}\Omega^{1}(\reg(X),h)$ for each $i\in \{1,...,n\}$ and for each $j\in \{1,...,m\}$ respectively. Summarizing, for each $j\in \{1,...,m\}$, we have
\begin{align}
\nonumber &\|d(\gamma_{p_j}|_{\reg(X)})\|_{L^{\infty}\Omega^1(\reg(X),h)}=\left\|\frac{\xi d(\beta'_{p_j}|_{\reg(X)})-\beta'_{p_j}d(\xi|_{\reg(X)})}{\xi^2}\right\|_{L^{\infty}\Omega^1(\reg(X),h)}\leq\\
\nonumber & \|\xi^{-2}|_{\reg(X)}\|_{L^{\infty}(\reg(X),h)}\|\xi|_{\reg(X)}\|_{L^{\infty}(\reg(X),h)} \|d(\beta'_{p_j}|_{\reg(X)})\|_{L^{\infty}\Omega^1(\reg(X),h)}+\\
\nonumber & \|\xi^{-2}|_{\reg(X)}\|_{L^{\infty}(\reg(X),h)}\|d(\xi|_{\reg(X)})\|_{L^{\infty}\Omega^{1}(\reg(X),h)} \|\beta'_{p_j}|_{\reg(X)}\|_{L^{\infty}(\reg(X),h)}<\infty
\end{align}
and thus \eqref{mcvf} is established.\\ We are finally in position to show the existence of a continuous inclusion $W^{1,2}(\reg(X),h)\hookrightarrow L^{\frac{2v}{v-1}}(\reg(X),h)$. To this aim, using the first point of this theorem, it is enough to prove that there exists a constant $r>0$ such that for every $f\in C^{\infty}_c(\reg(X))$  we have 
\begin{equation}
\label{disu}
\|f\|_{L^{\frac{2v}{v-1}}(\reg(X),h)}\leq r\|f\|_{W^{1,2}(\reg(X,h))}.
\end{equation}
Hence let us consider any $f\in C^{\infty}_c(\reg(X))$. Let us fix an open cover  $\mathcal{W}$ of $X$ and a subordinated continuous partition of unity $\{\gamma_{q_1},...,\gamma_{q_n},\gamma_{p_1},...,\gamma_{p_m}\}$ both as in \eqref{goodpart}. Obviously $f\gamma_{q_i}\in C^{\infty}_c(V_{q_i})$ for each $i\in \{1,...,n\}$ and $f\gamma_{p_j}\in C^{\infty}_c(\reg(V_{p_j}))$ for each $j\in \{1,...,m\}$.
According to \eqref{metz} there are positive constants $r_i>0$, $s_j>0$, $i\in \{1,...,n\}$, $j\in \{1,...,m\}$ such that: 
\begin{align}
\label{camelopardalis}
\nonumber & \|f\|_{L^{\frac{2v}{v-1}}(\reg(X),h)}=\|\sum_{i=1}^n\gamma_{q_i}f+\sum_{j=1}^m\gamma_{p_j}f\|_{L^{\frac{2v}{v-1}}(\reg(X),h)}\leq \|\sum_{i=1}^n\gamma_{q_i}f\|_{L^{\frac{2v}{v-1}}(\reg(X),h)}+\|\sum_{j=1}^m\gamma_{p_j}f\|_{L^{\frac{2v}{v-1}}(\reg(X),h)}=\\
\nonumber & \|\sum_{i=1}^n\gamma_{q_i}f\|_{L^{\frac{2v}{v-1}}(V_{q_i},h|_{V_{q_i}})}+\|\sum_{j=1}^m\gamma_{p_j}f\|_{L^{\frac{2v}{v-1}}(\reg(V_{p_j}),h|_{\reg(V_{p_j})})}\leq \sum_{i=1}^nr_i\|\gamma_{q_i}f\|_{W_0^{1,2}(V_{q_i},h|_{V_{q_i}})}+\\
\nonumber &  \sum_{j=1}^ms_j\|\gamma_{p_j}f\|_{W_0^{1,2}(\reg(V_{p_j}),h|_{\reg(V_{p_j})})}= \sum_{i=1}^nr_i\|\gamma_{q_i}f\|_{W_0^{1,2}(\reg(X),h)}+ \sum_{j=1}^ms_j\|\gamma_{p_j}f\|_{W_0^{1,2}(\reg(X),h)}\leq \\
\nonumber & \sum_{i=1}^nr_i\|\gamma_{q_i}f\|_{L^2(\reg(X),h)}+\sum_{i=1}^nr_i\|d(\gamma_{q_i}f)\|_{L^2\Omega^1(\reg(X),h)} +\sum_{j=1}^ms_j\|\gamma_{p_j}f\|_{L^2(\reg(X),h)}+\sum_{j=1}^ms_j\|d(\gamma_{p_j}f)\|_{L^2\Omega^1(\reg(X),h)}\leq\\
\nonumber & \sum_{i=1}^nr_i\|\gamma_{q_i}f\|_{L^2(\reg(X),h)}+\sum_{i=1}^nr_i\|\gamma_{q_i}df\|_{L^2\Omega^1(\reg(X),h)} +\sum_{j=1}^ms_j\|\gamma_{p_j}f\|_{L^2(\reg(X),h)}+\sum_{j=1}^ms_j\|\gamma_{p_j}df\|_{L^2\Omega^1(\reg(X),h)}+\\
\nonumber & \sum_{i=1}^nr_i\|fd\gamma_{q_i}\|_{L^2\Omega^1(\reg(X),h)}+\sum_{j=1}^ms_j\|fd(\gamma_{p_j}|_{\reg(X)})\|_{L^2\Omega^1(\reg(X),h)}\leq \\
 \nonumber & \sum_{i=1}^nr_i\|f\|_{L^2(\reg(X),h)}+\sum_{i=1}^nr_i\|df\|_{L^2\Omega^1(\reg(X),h)} +\sum_{j=1}^ms_j\|f\|_{L^2(\reg(X),h)}+\sum_{j=1}^ms_j\|df\|_{L^2\Omega^1(\reg(X),h)}+\\
\nonumber & \|f\|_{L^2(\reg(X),h)}\sum_{i=1}^nr_i\|d\gamma_{q_i}\|_{L^{\infty}\Omega^1(\reg(X),h)}+\|f\|_{L^2(\reg(X),h)}\sum_{j=1}^ms_j\|d(\gamma_{p_j}|_{\reg(X)})\|_{L^{\infty}\Omega^1(\reg(X),h)}\leq r\|f\|_{W^{1,2}(\reg(X),h)}
\end{align}
for some $r>0$ sufficiently big. The proof of the second point is thus complete. Finally the \emph{third property} of this  statement follows by the second one and Prop. 4.3 in \cite{FraBei}.
\end{proof}

The next result is concerned with the properties of the Friedrichs extension of the Laplace-Beltrami operator.

\begin{teo}
\label{Beltrami}
Let $(X,h)$ be a compact and irreducible Hermitian complex space of complex dimension $v>1$.  Consider the Laplace-Beltrami operator $\Delta:C^{\infty}_c(\reg(X))\rightarrow C^{\infty}_c(\reg(X))$ and let 
\begin{equation}
\label{fist}
\Delta^{\mathcal{F}}:L^2(\reg(X),h)\rightarrow L^2(\reg(X),h)
\end{equation}
be its Friedrichs extension. Then \eqref{fist} has discrete spectrum and 
\begin{equation}
\label{core}
\{f\in C^{\infty}(\reg(X)): f\in L^2(\reg(X),h), df\in L^2\Omega^1(\reg(X),h), d^t(df)\in L^2(\reg(X),h)\}
\end{equation}
 is a core domain for \eqref{fist}.\\Let $$0=\lambda_1< \lambda_2\leq...\leq \lambda_k\leq...$$ be the eigenvalues of \eqref{fist}. Then 
\begin{equation}
\label{asy}
\lim\inf \lambda_k k^{-\frac{1}{v}}>0
\end{equation}
as $k\rightarrow \infty$. Finally let 
\begin{equation}
\label{love}
e^{-t\Delta^{\mathcal{F}}}:L^2(\reg(X),h)\rightarrow L^2(\reg(X),h)
\end{equation}
be the heat operator associated to \eqref{fist}. Then \eqref{love} is a trace class operator and its trace satisfies the following estimate
\begin{equation}
\label{cisiamo}
\Tr(e^{-t\Delta^{\mathcal{F}}})\leq C\vol_h(\reg(X))t^{-v}
\end{equation}
 for $0<t\leq 1$ and for some constant $C>0$.
\end{teo}

\begin{proof}
According to Th. \ref{dollar} we have a continuous inclusion  $W^{1,2}(\reg(X),h)\hookrightarrow L^{\frac{2v}{v-1}}(\reg(X),h)$. Therefore, by Prop. 4.4 in \cite{FraBei}, we can conclude that \eqref{love} is a trace class operator and its trace satisfies the following estimate $$\Tr(e^{-t\Delta^{\mathcal{F}}})\leq C\vol_h(\reg(X))t^{-v}$$ for $0<t\leq 1$ and for some constant $C>0$. Since \eqref{love} is a trace class operator we get immediately that \eqref{fist} has discrete spectrum. Moreover, by the fact that  $\Tr(e^{-t\Delta^{\mathcal{F}}})=\sum_{k\in \mathbb{N}}e^{-t\lambda_k} \leq C\vol_h(\reg(X))t^{-v}$  for $t\in (0,1]$, we get, by applying a classical argument from Tauberian theory as in \cite{Tay} page 107, that $\lim\inf \lambda_k k^{-\frac{1}{v}}>0$ as $k\rightarrow \infty$.  Concerning \eqref{core} we know that $C^{\infty}(\reg(X))\cap \mathcal{D}(\Delta^{\mathcal{F}})$ is dense in $\mathcal{D}(\Delta^{\mathcal{F}})$, see for instance Prop. 2.1 in \cite{FraBei}. Moreover  $\Delta^{\mathcal{F}}=d^{t}_{\max}\circ d_{\min}$, see \cite{BLE} Lemma 3.1, and by the previous theorem we have $d_{\max}=d_{\min}$ which in turn implies that $d^{t}_{\max}=d^t_{\min}$. Thereby we have shown that $C^{\infty}(\reg(X))\cap \mathcal{D}(\Delta^{\mathcal{F}})=\{f\in C^{\infty}(\reg(X)): f\in L^2(\reg(X),h), df\in L^2\Omega^1(\reg(X),h), d^t(df)\in L^2(\reg(X),h)\}$ and thus we can conclude that \eqref{core} is a core domain for \eqref{fist}. Finally we observe that the third point of Th. \ref{dollar} implies directly that $\Delta^{\mathcal{F}}:L^2(\reg(X),h)\rightarrow L^2(\reg(X),h)$ has discrete spectrum. Indeed we have a continuous inclusion $\mathcal{D}(\Delta^{\mathcal{F}})\hookrightarrow W^{1,2}(\reg(X),h)$ and a compact inclusion $W^{1,2}(\reg(X),h)\hookrightarrow L^2(\reg(X),h)$ where $\mathcal{D}(\Delta^{\mathcal{F}})$ is endowed with the corresponding graph norm. Hence by composing we get a compact inclusion $\mathcal{D}(\Delta^{\mathcal{F}})\hookrightarrow L^2(\reg(X),h)$ and this in turn amounts to saying that  \eqref{fist} has discrete spectrum.
\end{proof}
The next results are concerned with the Friedrichs extension of the Hodge-Kodaira Laplacian acting on functions.

\begin{prop}
\label{estimate}
Let $(X,h)$ be a compact and irreducible Hermitian complex space. There exist positive constants $a$ and $b$ such that for each $f\in C^{\infty}_c(\reg(X))$ we have 
\begin{equation}
\label{aprile}
\|df\|^2_{L^2\Omega^1(\reg(X),h)}\leq a\|f\|^2_{L^2(\reg(X),h)}+b\|\overline{\partial}f\|^2_{L^2\Omega^{0,1}(\reg(X),h)}.
\end{equation}
\end{prop}

\begin{proof}
We start with some remarks about the notation that will be used through the proof. As we are going to use several metrics,  given a Hermitian manifold $(M,h)$, in order to avoid any confusion, we will label with $\Delta^h$ the corresponding Laplace-Beltrami operator. Analogously $\Delta^h_{\overline{\partial}}$ will stand for the associated Hodge-Kodaira Laplacian (acting on functions). We are now ready for the proof. As in the proof of Th. \ref{dollar} we can find a finite open cover of $X$, $\{V_1,...,V_s\}$, with a subordinated partition of unity, $\{\gamma_1,...,\gamma_s\}$, such that:
\begin{itemize}
\item the properties listed in \eqref{goodpart}-\eqref{mcvf} are fulfilled,
\item for each $i=1,...,s$ there exists a K\"ahler metric $h_i$ on $\reg(V_i)$ such that $h|_{\reg(V_i)}$ and $h_i$ are quasi-isometric.
\end{itemize}
Let $f\in C^{\infty}_c(\reg(X))$. Using the fact that $f\gamma_i\in C^{\infty}_c(\reg(V_i))$ for each $i=1,...,s$ and that for some $c>0$, $d>0$ and for each $i=1,...,s$ the following inequalities hold $\|\overline{\partial}(\gamma_i|_{\reg(X)})\|_{L^{\infty}\Omega^1(\reg(X),h)}\leq d$ and $c^{-1}\|\ \|^2_{L^2\Omega^1(\reg(V_i),h_i)}\leq\|\ \|^2_{L^2\Omega^{0,1}(\reg(V_i),h|_{\reg(V_i)})}\leq c\|\ \|^2_{L^2\Omega^1(\reg(V_i),h_i)}$, we have:
\begin{align}
\nonumber & \|df\|^2_{L^2\Omega^1(\reg(X),h)}=\|\sum_{i=1}^sd(f\gamma_i)\|^2_{L^2\Omega^1(\reg(X),h)}\leq s\sum_{i=1}^s\|d(f\gamma_i)\|^2_{L^2\Omega^1(\reg(X),h)}=s\sum_{i=1}^s\|d(f\gamma_i)\|^2_{L^2\Omega^1(\reg(V_i),h|_{V_i})}\leq\\
\nonumber & sc\sum_{i=1}^s\|d(f\gamma_i)\|^2_{L^2\Omega^1(\reg(V_i),h_i)}=sc\sum_{i=1}^s\langle \Delta^{h_i}(f\gamma_i),f\gamma_i\rangle_{L^2(\reg(V_i),h_i)}=2sc\sum_{i=1}^s\langle \Delta^{h_i}_{\overline{\partial}}(f\gamma_i),f\gamma_i\rangle_{L^2(\reg(V_i),h_i)}=\\
\nonumber & 2sc\sum_{i=1}^s\|\overline{\partial}(f\gamma_i)\|^2_{L^2\Omega^{0,1}(\reg(V_i),h_i)}\leq 2sc^2\sum_{i=1}^s\|\overline{\partial}(f\gamma_i)\|^2_{L^2\Omega^{0,1}(\reg(V_i),h|_{\reg(V_i)})}=2sc^2\sum_{i=1}^s\|\overline{\partial}(f\gamma_i)\|^2_{L^2\Omega^{0,1}(\reg(X),h)}=\\
\nonumber & 2sc^2\sum_{i=1}^s\|\gamma_i\overline{\partial}f+f\overline{\partial}(\gamma_i|_{\reg(V_i)})\|^2_{L^2\Omega^{0,1}(\reg(X),h)}\leq 2sc^2\sum_{i=1}^s2(\|\gamma_i\overline{\partial}f\|^2_{L^2\Omega^{0,1}(\reg(X),h)}+\|f\overline{\partial}(\gamma_i|_{\reg(V_i)})\|^2_{L^2\Omega^{0,1}(\reg(X),h)})=\\
\nonumber & 4sc^2\sum_{i=1}^s\|\gamma_i\overline{\partial}f\|^2_{L^2\Omega^{0,1}(\reg(X),h)}+4sc^2\sum_{i=1}^s\|f\overline{\partial}(\gamma_i|_{\reg(V_i)})\|^2_{L^2\Omega^{0,1}(\reg(X),h)}\leq 4s^2c^2\|\overline{\partial}f\|^2_{L^2\Omega^{0,1}(\reg(X),h)}+\\
\nonumber & 4sc^2\sum_{i=1}^s\|f\|^2_{L^2(\reg(X),h)}\|\overline{\partial}(\gamma_i|_{\reg(V_i)})\|^2_{L^{\infty}\Omega^{0,1}(\reg(X),h)}\leq 4s^2c^2\|\overline{\partial}f\|^2_{L^2\Omega^{0,1}(\reg(X),h)}+4s^2c^2d^2\|f\|^2_{L^2(\reg(X),h)}.
\end{align}
In conclusion we showed the existence of two positive constants $a=(2sc)^2$ and $b=(2scd)^2$ such that for each $f\in C^{\infty}_c(\reg(X))$ we have $$\|df\|^2_{L^2\Omega^1(\reg(X),h)}\leq a\|f\|^2_{L^2(\reg(X),h)}+b\|\overline{\partial}f\|^2_{L^2\Omega^{0,1}(\reg(X),h)}.$$ The proof is thus concluded.
\end{proof}

We have now the following theorems.

\begin{teo}
\label{hklaplacian}
Let $(X,h)$ be a compact and irreducible Hermitian complex space of complex dimension $v>1$. Consider the Hodge-Kodaira Laplacian acting on functions  $\Delta_{\overline{\partial}}:C^{\infty}_c(\reg(X))\rightarrow C^{\infty}_c(\reg(X))$ and let 
\begin{equation}
\label{wise}
\Delta^{\mathcal{F}}_{\overline{\partial}}:L^2(\reg(X),h)\rightarrow L^2(\reg(X),h)
\end{equation}
be its Friedrichs extension. We have the following properties:
\begin{enumerate}
\item Let $f\in L^{2}(\reg(X),h)$. Then $f\in \mathcal{D}(\overline{\partial}_{\min})$ if and only if $f\in W^{1,2}(\reg(X),h)$. Moreover the graph norm induced by $\overline{\partial}_{\min}$ and the norm of $W^{1,2}(\reg(X),h)$ are equivalent. Therefore, at the level of topological vector spaces, we have the following equality: $\mathcal{D}(\overline{\partial}_{\min})=W^{1,2}(\reg(X),h)$.
\item We have a continuous inclusion $\mathcal{D}(\overline{\partial}_{\min})\hookrightarrow L^{\frac{2v}{v-1}}(\reg(X),h)$ where $\mathcal{D}(\overline{\partial}_{\min})$ is endowed with the corresponding graph norm.
\item The natural inclusion $\mathcal{D}(\overline{\partial}_{\min})\hookrightarrow L^2(\reg(X),h)$ is a compact operator where $\mathcal{D}(\overline{\partial}_{\min})$ is endowed with the corresponding graph norm.
\item The operator $\Delta_{\overline{\partial}}^{\mathcal{F}}$ in \eqref{wise} has discrete spectrum.
\end{enumerate}
\end{teo}

\begin{proof}
Clearly given any $f\in C^{\infty}_c(\reg(X))$ we have $\|df\|^2_{L^2\Omega^1(\reg(X,h))}=\|\partial f\|^2_{L^2\Omega^{1,0}(\reg(X,h))}+\|\overline{\partial}f\|^2_{L^2\Omega^{0,1}(\reg(X,h))}$. Joining this observation with Prop.\ref{estimate} we have the first property of the above list. The second and the third properties follow by the first one and Th. \ref{dollar}.  Finally for the fourth property we argue as follows: according to \cite{BLE} Lemma 3.1 we have $\Delta^{\mathcal{F}}_{\overline{\partial}}=\overline{\partial}^t_{\max}\circ \overline{\partial}_{\min}$. Moreover it is easy to check that we have the inequality $$\|\overline{\partial}_{\min}f\|^2_{L^2\Omega^{0,1}(\reg(X),h)}\leq\frac{1}{2}\left(\|f\|_{L^2(\reg(X),h)}^2+\|\overline{\partial}^t_{\max}(\overline{\partial}_{\min}f)\|^2_{L^2(\reg(X),h)}\right)$$ for each $f\in \mathcal{D}(\Delta_{\overline{\partial}}^{\mathcal{F}})$. Hence the above inequality tells us that we have a continuous inclusion $\mathcal{D}(\Delta^{\mathcal{F}}_{\overline{\partial}})\hookrightarrow \mathcal{D}(\overline{\partial}_{\min})$ where each domain is endowed with the corresponding graph norm. Composing the latter inclusion with the inclusion $\mathcal{D}(\overline{\partial}_{\min})\hookrightarrow L^2(\reg(X),h)$ we have, thanks to third point above, that the inclusion  $\mathcal{D}(\Delta^{\mathcal{F}}_{\overline{\partial}})\hookrightarrow L^2(\reg(X),h)$ is actually a compact inclusion and this is well known to be equivalent to the fact that \eqref{wise} has discrete spectrum. The proof is thus complete.
\end{proof}

\begin{teo}
\label{spechk}
In the setting of Th. \ref{hklaplacian}. Let $(a,b)$ be any pair of positive numbers such that \eqref{aprile} holds true. Let $$0=\lambda_1<\lambda_2\leq\lambda_3\leq...\quad\quad \text{and}\quad\quad  0=\mu_1<\mu_2\leq \mu_3\leq...$$ be the eigenvalues of \eqref{fist} and \eqref{wise} respectively\footnote{  $0=\mu_1<\mu_2$ is a consequence of the fact that $\ker(\overline{\partial}_{\min})=\mathbb{C}$ as showed in \cite{JRu} Th. 1.6 or \cite{PS} Th. A.}. Then we have the following inequalities
\begin{equation}
\label{orociok}
\mu_k\leq \lambda_k \leq a+b\mu_k
\end{equation}
for each $k\in \mathbb{N}$.
As a first consequence we have 
\begin{equation}
\label{asilo}
\lim\inf \mu_k k^{-\frac{1}{v}}>0
\end{equation}
as $k\rightarrow \infty$. Finally let 
\begin{equation}
\label{lotto}
e^{-t\Delta_{\overline{\partial}}^{\mathcal{F}}}:L^2(\reg(X),h)\rightarrow L^2(\reg(X),h)
\end{equation}
be the heat operator associated to \eqref{wise}. Then \eqref{lotto} is a trace class operator and its trace satisfies the following estimates
\begin{equation}
\label{matriciana}
e^{-at}\Tr(e^{-tb\Delta_{\overline{\partial}}^{\mathcal{F}}})\leq \Tr(e^{-t\Delta^{\mathcal{F}}})\leq \Tr(e^{-t\Delta_{\overline{\partial}}^{\mathcal{F}}})
\end{equation}
and, for each $0<t\leq 1$,
\begin{equation}
\label{agamennone}
\Tr(e^{-t\Delta_{\overline{\partial}}^{\mathcal{F}}})\leq r\vol_h(\reg(X))t^{-v}
\end{equation}
where $r$ is some positive constant.
\end{teo}

\begin{proof}
According to the min-max theorem, see for instance \cite{KSC}, we can characterize the eigenvalues of \eqref{fist} as 
$$\lambda_k=\inf_{F\in \mathfrak{F}_k\cap\mathcal{D}(\Delta^{\mathcal{F}})}\sup_{f\in F, f\neq 0}\frac{\langle\Delta^{\mathcal{F}}f,f\rangle_{L^2(\reg(X),h)}}{\|f\|^2_{L^2(\reg(X),h)}}=\inf_{F\in \mathfrak{F}_k\cap\mathcal{D}(d_{\min})}\sup_{f\in F, f\neq 0}\frac{\langle d_{\min}f,d_{\min}f\rangle_{L^2\Omega^1(\reg(X),h)}}{\|f\|^2_{L^2(\reg(X),h)}}$$
where $\mathfrak{F}_k$ denotes the set of linear subspaces of $L^2(\reg(X),h)$ of dimension at most $k$. Analogously for the eigenvalues of \eqref{wise} we have 
\begin{align}
\label{mini}
& \mu_k=\inf_{F\in \mathfrak{F}_k\cap\mathcal{D}(\Delta_{\overline{\partial}}^{\mathcal{F}})}\sup_{f\in F, f\neq 0}\frac{\langle \Delta_{\overline{\partial}}^{\mathcal{F}}f,f\rangle_{L^2(\reg(X),h)}}{\|f\|^2_{L^2(\reg(X),h)}}= \inf_{F\in \mathfrak{F}_k\cap\mathcal{D}(\overline{\partial}_{\min})}\sup_{f\in F, f\neq 0} \frac{\langle\overline{\partial}_{\min}f,\overline{\partial}_{\min}f\rangle_{L^2\Omega^{0,1}(\reg(X),h)}}{\|f\|^2_{L^2(\reg(X),h)}}.
\end{align}
Consider  an orthonormal basis of $L^2(\reg(X),h)$, $\{\phi_n,\ n\in \mathbb{N}\}$, made by eigensections of \eqref{fist} such that $\Delta^{\mathcal{F}}\phi_n=\lambda_n\phi_n$. Let us define  $F_k\in \mathfrak{F}_k$ as the $k$-dimensional subspace of $L^2(\reg(X),h)$ generated by $\{\phi_1,...,\phi_k\}$. Then, see for instance \cite{KSC} page 279, we have $$\lambda_k=\sup_{f\in F_k, f\neq 0}\frac{\langle d_{\min}f,d_{\min}f\rangle_{L^2\Omega^1(\reg(X),h)}}{\|f\|^2_{L^2(\reg(X),h)}}.$$
Hence we have 
\begin{align}
\nonumber & \lambda_k= \sup_{f\in F_k, f\neq 0}\frac{\langle d_{\min}f,d_{\min}f\rangle_{L^2\Omega^1(\reg(X),h)}}{\|f\|^2_{L^2(\reg(X),h)}}\geq  \sup_{f\in F_k, f\neq 0}\frac{\langle\overline{\partial}_{\min}f,\overline{\partial}_{\min}f\rangle_{L^2\Omega^{0,1}(\reg(X),h)}}{\|f\|^2_{L^2(\reg(X),h)}}\\
\nonumber & \geq  \inf_{F\in \mathfrak{F}_k\cap\mathcal{D}(\overline{\partial}_{\min})}\sup_{f\in F, f\neq 0}\frac{\langle \overline{\partial}_{\min}f,\overline{\partial}_{\min}f\rangle_{L^2\Omega^{0,1}(\reg(X),h)}}{\|f\|^2_{L^2(\reg(X),h)}}=\mu_k.
\end{align}
This establishes the first part of  \eqref{orociok}. In order to establish the other inequality let $\{\psi_n,\ n\in \mathbb{N}\}$ be  an orthonormal basis of $L^2(\reg(X),h)$ made by eigensections of \eqref{wise} such that $\Delta_{\overline{\partial}}^{\mathcal{F}}\psi_n=\mu_n\psi_n$. Similarly to the previous case let us define $G_k\in \mathfrak{F}_k$ as the $k$-dimensional subspace of $L^2(\reg(X),h)$ generated by $\{\psi_1,...,\psi_k\}$. Let $(a,b)$ be a pair of positive constants such that \eqref{aprile} holds true. Then, by \eqref{aprile} and Th. \ref{hklaplacian}, we have 
\begin{align}
\nonumber & a+b\mu_k=\\
\nonumber & a+b\left(\sup_{f\in G_k, f\neq 0}\frac{\langle \overline{\partial}_{\min}f,\overline{\partial}_{\min}f\rangle_{L^2\Omega^{0,1}(\reg(X),h)}}{\|f\|^2_{L^2(\reg(X),h)}}\right)=\sup_{f\in G_k, f\neq 0}\frac{a\|f\|^2_{L^2(\reg(X),h)}+b\langle\overline{\partial}_{\min}f,\overline{\partial}_{\min}f\rangle_{L^2\Omega^{0,1}(\reg(X),h)}}{\|f\|^2_{L^2(\reg(X),h)}}\\
\nonumber & \geq \sup_{f\in G_k, f\neq 0}\frac{\langle d_{\min}f,d_{\min}f\rangle_{L^2\Omega^1(\reg(X),h)}}{\|f\|^2_{L^2(\reg(X),h)}} \geq \inf_{F\in \mathfrak{F}_k\cap\mathcal{D}(d_{\min})}\sup_{f\in F, f\neq 0}\frac{\langle d_{\min}f,d_{\min}f\rangle_{L^2\Omega^1(\reg(X),h)}}{\|f\|^2_{L^2(\reg(X),h)}}=\lambda_k.
\end{align}
This establishes \eqref{orociok}. Property \eqref{asilo}  follows now easily by \eqref{orociok} and Th. \ref{Beltrami}. Indeed $\lim\inf \mu_k k^{-\frac{1}{v}}>0$ as $k\rightarrow \infty$ if and only if  $\lim\inf (a+b\mu_k) k^{-\frac{1}{v}}>0$ as $k\rightarrow \infty$ and this last inequality follows by \eqref{asy} and \eqref{orociok}. The asymptotic inequality \eqref{asilo} implies that $\sum_{k\in \mathbb{N}} e^{-t\mu_k}<\infty$ and this allows us to conclude that \eqref{lotto} is a trace class operator. For each $t>0$, we have $\sum_{k\in \mathbb{N}}e^{-t(a+b\mu_k)}\leq \sum_{k\in \mathbb{N}}e^{-t\lambda_k}\leq \sum_{k\in \mathbb{N}}e^{-t\mu_k}$ which in turn implies  that  $$e^{-at}\Tr(e^{-tb\Delta_{\overline{\partial}}^{\mathcal{F}}})\leq \Tr(e^{-t\Delta^{\mathcal{F}}})\leq \Tr(e^{-t\Delta_{\overline{\partial}}^{\mathcal{F}}}).$$ Finally \eqref{agamennone} follows immediately by \eqref{cisiamo} and \eqref{matriciana}. 
\end{proof}

In the remaining part of this section we collect some consequences of the previous theorems.

\begin{cor}
\label{attore}
Let $(X,h)$ be as in Th. \ref{dollar}. Let $E\rightarrow X$ be a holomorphic vector bundle of complex rank $s$. Let $\rho$ be a Hermitian metric on $E|_{\reg(X)}$. Assume that for each point $p\in X$ there exists an open neighborhood $U$, a positive constant $c$ and a holomorphic trivialization $\psi:E|_U\rightarrow U\times \mathbb{C}^s$ such that, labeling by $\sigma$ the Hermitian metric on $\reg(U)\times \mathbb{C}^s$ induced by $\rho$ through $\psi$, we have 
\begin{equation}
\label{trappolo}
c^{-1}\tau\leq \sigma\leq c\tau
\end{equation}
 where $\tau$ is the Hermitian metric on $\reg(U)\times \mathbb{C}^s$ that assigns to each point of $\reg(U)$ the standard K\"ahler metric of $\mathbb{C}^s$.
Let $\overline{\partial}_E:C^{\infty}_c(\reg(X),E|_{\reg(X)})\rightarrow \Omega^{0,1}_c(\reg(X),E|_{\reg(X)})$ be the associated Dolbeault operator and let $$\overline{\partial}_{E,\min}:L^2(\reg(X),E|_{\reg(X)})\rightarrow L^2\Omega^{0,1}(\reg(X),E|_{\reg(X)})$$ be its minimal extension. We have the following properties:
\begin{enumerate}
\item The inclusion $\mathcal{D}(\overline{\partial}_{E,\min})\hookrightarrow L^2(\reg(X),E|_{\reg(X)})$ is a compact operator, where $\mathcal{D}(\overline{\partial}_{E,\min})$ is endowed with the corresponding graph norm.
\item The operator $\overline{\partial}_{E,\max}^t\circ\overline{\partial}_{E,\min}:L^2(\reg(X),E|_{\reg(X)})\rightarrow L^2(\reg(X),E|_{\reg(X)})$ has discrete spectrum.
\end{enumerate}
\end{cor}
\begin{proof}
As for instance  remarked in the proof of Th. \ref{hklaplacian} it is enough to show only the first property. Let $\{s_n\}_{n\in \mathbb{N}}$ be any bounded sequence in $\mathcal{D}(\overline{\partial}_{E,\min})$. Without loss of generality we can assume that $\{s_n\}_{n\in \mathbb{N}}\subset C^{\infty}_c(\reg(X),E)$. Let us start by  considering an open cover $\mathcal{W}:=\{V_{q_1},...,V_{q_n},V_{p_1},...,V_{p_m}\}$ with a subordinated continuous  partition of unity $\{\gamma_{q_1},...,\gamma_{q_n},\gamma_{p_1},...,\gamma_{p_m}\}$ both as in \eqref{goodpart}.  We recall that $\overline{V_{q_i}}\cap \sing(X)=\emptyset$ for $i=1,...,n$. Moreover, as in the proof of Th. \ref{hklaplacian}, we can assume that $\mathcal{W}$ is made in such a way that for each $i\in \{1,...,n\}$ there exists a relatively compact open subset $B_{q_i}\subset \mathbb{C}^v$ with smooth boundary and a biholomorphism  $\psi_i:V_{q_i}\rightarrow B_{q_i}$ that extends to a diffeomorphism $\psi_i:\overline{V_{q_i}}\rightarrow \overline{B_{q_i}}$. Additionally, without loss of generality, we can assume that the integers $n$ and $m$, the points $q_1,...,q_n,p_1,...,p_m$ and the corresponding neighborhoods $U_{q_1},...,U_{q_n},V_{p_1},...,V_{p_m}$ are chosen  in such a way that there exist open neighborhoods with compact closure $\tilde{V}_{q_1},...,\tilde{V}_{q_n},\tilde{V}_{p_1},...,\tilde{V}_{p_m}$ such that $\overline{V_{q_i}}\subset \tilde{V}_{q_i}$, there is a holomorphic trivialization $\psi_{q_i}:E|_{\tilde{V}_{q_i}}\rightarrow \tilde{V}_{q_i}\times \mathbb{C}^s$, $\overline{V_{p_j}}\subset \tilde{V}_{p_j}$ and there is a holomorphic trivialization $\psi_{p_j}:E|_{\tilde{V}_{p_j}}\rightarrow \tilde{V}_{p_j}\times \mathbb{C}^s$ that satisfies \eqref{trappolo}. Consider now the sequence $\{\gamma_{p_j}s_n\}_{n\in \mathbb{N}}\subset C^{\infty}_c(\reg(V_{p_j}),E|_{\reg(V_{p_j})})$ for an arbitrary $j\in \{1,...,m\}$. Let $\psi_{p_j}:E|_{V_{p_j}}\rightarrow V_{p_j}\times \mathbb{C}^s$ be a holomorphic trivialization that obeys \eqref{trappolo}. Let $\{\mathfrak{f}_{p_j,n}\}_{n\in \mathbb{N}}\subset C^{\infty}_c(\reg(V_{p_j}),\mathbb{C}^s)$ be the sequence given by  $\mathfrak{f}_{p_j,n}:=\psi_{p_j}\circ (\gamma_{p_j}s_n)$. For each $n\in \mathbb{N}$ let  $(f_{p_j,1,n},...,f_{p_j,s,n})$  be the $s$ components of $\mathfrak{f}_{p_j,n}$. By the assumptions, in particular \eqref{mcvf} and \eqref{trappolo}, we have that $\{\mathfrak{f}_{p_j,n}\}_{n\in \mathbb{N}}$ is a bounded sequence in $\mathcal{D}(\overline{\partial}_{\min})$, the domain of 
\begin{equation}
\label{mascarus}
\overline{\partial}_{\min}:L^2(\reg(V_{p_j}),  \mathbb{C}^s)\rightarrow L^2(\reg(V_{p_j}), \Lambda^{0,1}(\reg(V_{p_j}))\otimes  \mathbb{C}^s)
\end{equation}
where in \eqref{mascarus} $\Lambda^{0,1}(\reg(V_{p_j}))$ is endowed with the Hermitian metric induced by $h$ and $\reg(V_{p_j})\times \mathbb{C}^s$ is endowed with $\tau$, see \eqref{trappolo}. The fact that $\{\mathfrak{f}_{p_j,n}\}_{n\in \mathbb{N}}$ is a bounded sequence in the domain of \eqref{mascarus} implies in turn that  $\{f_{p_j,1,n}\}_{n\in \mathbb{N}}$,...,$\{f_{p_j,s,n}\}_{n\in \mathbb{N}}$ are $s$ bounded sequences in the domain of 
\begin{equation}
\label{mascaruz}
\overline{\partial}_{\min}:L^2(\reg(V_{p_j}),h|_{\reg(V_{p_j})})\rightarrow L^2\Omega^{0,1}(\reg(V_{p_j}),h|_{\reg(V_{p_j})}).
\end{equation}
Therefore as a consequence of Th. \ref{hklaplacian} we can deduce the existence of a countable subset  $A\subset \mathbb{N}$ such that $\{f_{p_j,1,n}\}_{n\in A}$,...,$\{f_{p_j,s,n}\}_{n\in A}$ are $s$ convergent sequences in $L^2(\reg(V_{p_j}),h|_{\reg(V_{p_j})})$. Clearly this is equivalent to saying that $\{(f_{p_j,1,n},...,f_{p_j,s,n})\}_{n\in A}$ is a convergent sequence in $L^2(\reg(V_{p_j}),\mathbb{C}^s)$, where $\reg(V_{p_j})\times\mathbb{C}^s$ is endowed with $\tau$, and finally this amounts to saying that $\{\gamma_{p_j}s_n\}_{n\in A}$ is a convergent sequence in $L^2(\reg(V_{p_j}),E|_{\reg(V_{p_j})})$.  Clearly arguing in the same way we can show the existence of a convergent subsequence of $\{\gamma_{q_i}s_n\}_{n\in \mathbb{N}}$ for each $i\in \{1,...,n\}$. Summarizing we have proved the existence of an open cover $\mathcal{W}=\{V_{q_1},...,V_{q_n},V_{p_1},...,V_{p_m}\}$ of $X$ with a subordinated continuous partition of unity $\{\gamma_{q_1},...,\gamma_{q_n},\gamma_{p_1},...,\gamma_{p_m}\}$ such that $\{\gamma_{p_j}s_n\}_{n\in \mathbb{N}}$ admits a convergent subsequence in $L^2(\reg(X),E|_{\reg(X)})$ for each $j\in \{1,...,m\}$ and analogously $\{\gamma_{q_i}s_n\}_{n\in \mathbb{N}}$ admits a convergent subsequence in $L^2(\reg(X),E|_{\reg(X)})$ for each $i\in \{1,...,n\}$. It is now straightforward to check that $\{s_n\}_{n\in \mathbb{N}}$ admits a convergent subsequent in $L^2(\reg(X),E|_{\reg(X)})$. The proof is thus concluded.
\end{proof}

We add the following remark to the previous corollary. Let $(X,h)$ and $p:E\rightarrow X$ be as in Cor. \ref{attore}. Let $\pi:M\rightarrow X$ be a resolution of $X$. Let $F:=p^*E$ and let $\upsilon$ be any Hermitian metric on $F$. Finally let us define $\rho$ as the Hermitian metric on $E|_{\reg(X)}$ such that $(\pi|_{M\setminus D})^*\rho=\upsilon|_{M\setminus D}$ where $D=\pi^{-1}(\sing(X))$. Then it is easy to check that $(X,E,\rho)$ satisfy the assumptions of Cor. \ref{attore}.

\begin{cor}
\label{pino}
Let $(X,h)$ be as in Th. \ref{dollar}. Let $E$ be a vector bundle over $\reg(X)$ endowed with a metric  $\tau$. Finally let $\nabla:C^{\infty}(\reg(V),E)\rightarrow C^{\infty}(\reg(V),T^*\reg(X)\otimes E)$ be a metric connection. We have the following properties:
\begin{itemize}
\item $W^{1,2}(\reg(X),E)=W^{1,2}_0(\reg(X),E)$.
\item Assume that $v>1$. Then there exists a continuous inclusion $W^{1,2}(\reg(X),E)\hookrightarrow L^{\frac{2v}{v-1}}(\reg(X),E)$.
\item Assume that $v>1$. Then the inclusion $W^{1,2}(\reg(X),E)\hookrightarrow L^2(\reg(X),E)$ is a compact operator.
\end{itemize}
\end{cor}

\begin{proof}
According to \cite{BeGu} or \cite{JRU} we know that $(\reg(X),h)$  is  parabolic. Now the statement follows by Th. \ref{dollar} and Prop. 4.1 in \cite{FraBei}.
\end{proof}

\begin{cor}
\label{nove}
Let $(X,h)$ be as in Theorem \ref{dollar}. Let  $E$ and $F$ be two vector bundles over $\reg(X)$ endowed respectively with  metrics $\tau$ and $\rho$. Finally let  $\nabla:C^{\infty}(\reg(X),E)\rightarrow C^{\infty}(M,T^*\reg(X)\otimes E)$
be a metric connection. Consider a  first order differential operator of this type:
\begin{equation}
\label{gazzzz}
D:=\theta_0\circ \nabla:C^{\infty}_c(\reg(X),E)\rightarrow C^{\infty}_c(\reg(X),F)
\end{equation}
where $\theta_0\in  C^{\infty}(\reg(X),\Hom(T^*\reg(X)\otimes E,F)).$ Assume that  $\theta_0$ extends as a bounded operator  $$\theta:L^2(\reg(X), T^*\reg(X)\otimes E)\rightarrow L^2(\reg(X), F).$$  Then we have the following inclusion:
\begin{equation}
\label{cantianobello}
\mathcal{D}(D_{\max})\cap L^{\infty}(\reg(X),E)\subset \mathcal{D}(D_{\min}).
\end{equation}
In particular \eqref{cantianobello} holds when $D$ is the de Rham differential $d_k:\Omega_{c}^k(\reg(X))\rightarrow \Omega^{k+1}_c(\reg(X))$, a Dirac  operator   $D:C^{\infty}_c(\reg(X),E)\rightarrow C^{\infty}_c(\reg(X),E)$ or the Dolbeault operator $\overline{\partial}_{p,q}:\Omega_{c}^{p,q}(\reg(X))\rightarrow \Omega^{p,q+1}_c(\reg(X))$.
\end{cor}

\begin{proof}
This follows by the fact that $(\reg(X),h)$ is parabolic and Prop. 4.2 in \cite{FraBei}
\end{proof}

We point out that in the above corollary the  case of the Dolbeault operator has been already treated in \cite{JRU}. Consider again a compact and irreducible Hermitian complex space $(X,h)$ of complex dimension $v$. Let $\reg(X)$ be its regular part and let $E$ be a vector bundle over $\reg(X)$ endowed with a metric $\sigma$. Finally let $\nabla:C^{\infty}(\reg(V),E)\rightarrow C^{\infty}(\reg(V),T^*M\otimes E)$ be a metric connection.  We consider  a Schr\"odinger-type operators
\begin{equation}
\label{polvere}
\nabla^t\circ\nabla +L 
\end{equation}
 where $\nabla^t: C_c^{\infty}(\reg(X),T^*M\otimes E)\rightarrow C_c^{\infty}(\reg(X),E)$ is the formal adjoint of $\nabla$ and $L\in C^{\infty}(\reg(X),\End(E))$.

\begin{cor}
\label{skernel}
Let $V$, $E$, $\sigma$, $h$, $L$ and $\nabla$ be as described above. Assume $v>1$. Let $$P:=\nabla^t\circ \nabla +L,\ P:C^{\infty}_c(\reg(X),E)\rightarrow C_c^{\infty}(\reg(X),E)$$ be a Schr\"odinger type operator. Assume that:
\begin{itemize}
\item $P$ is symmetric and  positive.
\item    There is a constant $c\in \mathbb{R}$ such that, for each $s\in C^{\infty}(\reg(V),E)$, we have  $$\sigma(Ls,s)\geq c\sigma(s,s).$$
\end{itemize}
Let $P^{\mathcal{F}}:L^2(\reg(X),E)\rightarrow L^2(\reg(X),E)$ be the Friedrichs extension of $P$ and let  $\Delta^{\mathcal{F}}:L^2(\reg(X),g)\rightarrow L^2(\reg(X),g)$  be the Friedrichs extension of $\Delta:C^{\infty}_c(\reg(X))\rightarrow C^{\infty}_c(\reg(X))$. Then the heat operator associated to $P^{\mathcal{F}}$ $$e^{-tP^{\mathcal{F}}}:L^{2}(\reg(X),E)\longrightarrow L^2(\reg(X),E)$$  
is a trace class operator and its trace satisfies the following inequality: 
\begin{equation}
\label{marz}
\Tr(e^{-tP^{\mathcal{F}}})\leq me^{-tc}\Tr(e^{-t\Delta^{\mathcal{F}}}).
\end{equation}
where $m$ is the rank of the vector bundle $E$.
\end{cor}

\begin{proof}
This follows by Th. \ref{dollar} and Prop. 4.5 in \cite{FraBei}.
\end{proof}

\begin{cor}
\label{quovadis}
In the setting of Cor. \ref{skernel}. The operator  $P^{\mathcal{F}}:L^2(\reg(X),E)\rightarrow L^2(\reg(X),E)$  has  discrete spectrum.  Moreover, for $t\in (0,1]$, we have the following inequality:
\begin{equation}
\label{zun}
\Tr(e^{-tP^{\mathcal{F}}})\leq m\vol_h(\reg(X))Ce^{-tc}t^{-v}
\end{equation}
where $C$ is the same constant appearing in \eqref{cisiamo}. Let $\{\lambda_k\}$ be the sequence of eigenvalues of  $P^{\mathcal{F}}:L^2(\reg(V),E)\rightarrow L^2(\reg(V),E)$. Then we have  the following asymptotic inequality:
\begin{equation}
\label{waser}
\lim\inf \lambda_kk^{-\frac{1}{v}}>0
\end{equation}
as $k\rightarrow \infty$.
\end{cor}

\begin{proof}
This follows immediately by Cor. \ref{skernel} and  Th. \ref{Beltrami}.
\end{proof}

\begin{cor}
\label{pietrob}
Under the assumptions of Cor. \ref{skernel}.  Then we have the following properties:
\begin{enumerate}
\item $e^{-tP^{\mathcal{F}}}$ is a ultracontractive operator for each $0<t\leq 1$. This means that for each $0<t\leq 1$ there exists a positive constant $C_t>0$  such that   $$\|e^{-tP^{\mathcal{F}}}s\|_{L^{\infty}(\reg(X),E)}\leq C_t\|s\|_{L^1(\reg(X),E)}$$ for each $s\in L^{1}(\reg(X),E)$. In particular, for each $0<t\leq 1$, $e^{-tP^{\mathcal{F}}}:L^{1}(\reg(X),E)\rightarrow L^{\infty}(\reg(X),E)$ is continuous.
\item If $s$ is an eigensection of $P^\mathcal{F}:L^2(\reg(X),E)\rightarrow L^2(\reg(X),E)$ then $s\in L^{\infty}(\reg(X),E)$.
\end{enumerate}
\end{cor}
\begin{proof}
This follows by Cor. \ref{skernel} and Prop. 4.6 in \cite{FraBei}.
\end{proof}

We conclude this section with the following application.
\begin{cor}
\label{ibracadabra}
In the setting of Th. \ref{dollar}. Let $\phi$ be an eigenfunction of $\Delta^{\mathcal{F}}:L^2(\reg(X),h)\rightarrow L^2(\reg(X),h)$. Then $\phi\in L^{\infty}(\reg(X))$. 
\end{cor}

\begin{proof}
This follows immediately by Cor. \ref{pietrob}.
\end{proof}

\section{The Hodge-Kodaira Laplacian on possibly singular complex surfaces}

In this section we extend to the setting of compact and irreducible Hermitian complex spaces of complex dimension $2$ some of the results proved in \cite{FBei} in the case of complex projective surfaces endowed with the Fubini-Study metric. Since many of the following results follow by the same strategy used in \cite{FBei} we will omit the details in the proofs and we will refer directly to the corresponding results in \cite{FBei}.

\begin{teo}
\label{lillottina}
Let $(X,h)$ be a compact and irreducible Hermitian complex  space of complex dimension $v=2$. For each $q=0,1,2$  we have the following properties:
\begin{enumerate}
\item $\Delta_{\overline{\partial},2,q,\abs}:L^2\Omega^{2,q}(\reg(X),h)\rightarrow L^2\Omega^{2,q}(\reg(X),h)$ has discrete spectrum.
\item $\overline{\partial}_{2,\max}+\overline{\partial}^t_{2,\min}:L^2\Omega^{2,\bullet}(\reg(X),h)\rightarrow L^2\Omega^{2,\bullet}(\reg(X),h)$ has discrete spectrum.
\item $\Delta_{\overline{\partial},2,q}^{\mathcal{F}}:L^2\Omega^{2,q}(\reg(X),h)\rightarrow L^2\Omega^{2,q}(\reg(X),h)$ has discrete spectrum.
\end{enumerate}
\end{teo}

\begin{proof}
We start by  considering  the operator  $\Delta_{\overline{\partial},2,0,\abs}:L^2\Omega^{2,0}(\reg(V),h)\rightarrow L^2\Omega^{2,0}(\reg(V),h)$. In this case the statement is  a particular case of  Th. 5.1 in \cite{FBei}. Now we deal with $\Delta_{\overline{\partial},2,2,\abs}:L^2\Omega^{2,2}(\reg(X),h)\rightarrow L^2\Omega^{2,2}(\reg(X),h)$.  By Prop. \ref{occhiodibue} we have $*(c(\mathcal{D}(\Delta_{\overline{\partial},2,2,\abs})))=\mathcal{D}(\Delta_{\overline{\partial},\rel})$ and  $*\circ c\circ \Delta_{\overline{\partial},2,2,\abs}=\Delta_{\overline{\partial},\rel}\circ *\circ c.$
We are therefore left to prove that $\Delta_{\overline{\partial},\rel}:L^2(\reg(X),h)\rightarrow L^2(\reg(X),h)$ has discrete spectrum. This  follows by Th. \ref{hklaplacian} because when $(p,q)=(0,0)$ we have $\Delta_{\overline{\partial},\rel}=\overline{\partial}_{\max}^t\circ \overline{\partial}_{\min}=\Delta_{\overline{\partial}}^{\mathcal{F}}$. It remains to show that $\Delta_{\overline{\partial},2,1,\abs}:L^2\Omega^{2,1}(\reg(X),h)\rightarrow L^2\Omega^{2,1}(\reg(X),h)$ has discrete spectrum. This in turn is equivalent to showing that the inclusion
\begin{equation}
\label{compactww}
\mathcal{D}(\Delta_{\overline{\partial},2,1,  \abs})\hookrightarrow L^2\Omega^{2,1}(\reg(X),h)
\end{equation}
 is a compact operator where $\mathcal{D}(\Delta_{\overline{\partial},2,1,\abs})$ is endowed with the corresponding graph norm. According to \cite{JRu}  we know that $H^{2,q}_{2,\overline{\partial}_{\max}}(\reg(X),h)$ is finite dimensional for each $q$. Therefore, using  \cite{BL} Th. 2.4, we can conclude that, for each $q=0,...,2$, $\im(\overline{\partial}_{2,q,\max})$ is closed  and that $\Delta_{\overline{\partial},2,q,  \abs}:L^2\Omega^{2,q}(\reg(X),h)\rightarrow L^2\Omega^{2,q}(\reg(X),h)$ is a Fredholm operator on its domain endowed with the graph norm. Hence, by the fact that  $\Delta_{\overline{\partial},2,1, \abs}:L^2\Omega^{2,1}(\reg(X),h)\rightarrow L^2\Omega^{2,1}(\reg(X),h)$ is Fredholm and self-adjoint, we know now that \eqref{compactww} is a compact operator if and only if the following inclusion is a compact operator  
\begin{equation}
\label{scompactIIIq}
\left(\mathcal{D}(\Delta_{\overline{\partial},2,1, \abs})\cap \im(\Delta_{\overline{\partial},2,1,\abs})\right)\hookrightarrow L^2\Omega^{2,1}(\reg(X),h)
\end{equation}
 where  $\left(\mathcal{D}(\Delta_{\overline{\partial},2,1,\abs})\cap \im(\Delta_{\overline{\partial},2,1,\abs})\right)$ is endowed with the  graph norm of $\Delta_{\overline{\partial},2,1, \abs}$. Now the conclusion follows by arguing as in the proof of the first point of Th. 5.3 in \cite{FBei}. Finally the second and the third point of this theorem follow by using the same arguments used to show the second and the third point  of Th. 5.2 in \cite{FBei}.
\end{proof}

\begin{teo}
\label{coccabelladezio}
In the  setting of Th. \ref{lillottina}. Let $q\in \{0,1,2\}$ and consider the operator
\begin{equation}
\label{resile}
\Delta_{\overline{\partial},2,q,\abs}:L^2\Omega^{2,q}(\reg(X),h)\rightarrow L^2\Omega^{2,q}(\reg(X),h).
\end{equation}
Let $$0\leq \lambda_1\leq \lambda_2\leq...\leq \lambda_k\leq...$$ be the eigenvalues of \eqref{resile}. Then we have the following asymptotic inequality
\begin{equation}
\label{resiles}
\lim \inf \lambda_k k^{-\frac{1}{2}}>0
\end{equation}
as $k\rightarrow \infty$.\\ Consider now the heat operator associated to \eqref{resile}
\begin{equation}
\label{resilex}
e^{-t\Delta_{\overline{\partial},2,q,\abs}}:L^2\Omega^{2,q}(\reg(X),h)\rightarrow L^2\Omega^{2,q}(\reg(X),h).
\end{equation}
Then \eqref{resilex} is a trace class operator and its trace satisfies the following estimate
\begin{equation}
\label{resilez}
\Tr(e^{-t\Delta_{\overline{\partial},2,q,\abs}})\leq C_qt^{-2}
\end{equation}
for $t\in (0,1]$ and some constant $C_q>0$.
\end{teo}

\begin{proof}
The case $q=0$  follows by Th. 5.1 in \cite{FBei}. Consider now the case $q=2$. Then, as pointed out in the proof of Th. \ref{lillottina}, we have  $*\circ c \circ \Delta_{\overline{\partial},2,2, \abs}=\Delta^{\mathcal{F}}_{\overline{\partial}}\circ *\circ c$.  Therefore the case $q=2$ follows by Th. \ref{spechk}. Finally the case $q=1$ follows by using the same arguments used in the proof of Th. 5.4 in \cite{FBei}.
\end{proof}

As a consequence of the previous theorem we retrieve the McKean-Singer formula for the $L^2$-$\overline{\partial}$-complex $(L^2\Omega^{2,q}(\reg(X),h),\overline{\partial}_{2,q,\max})$ in the setting of compact and irreducible Hermitian complex spaces of complex dimension $2$. Consider the operator
\begin{equation}
\label{isabella}
\overline{\partial}_{2,0,\max}+\overline{\partial}^t_{2,1,\min}:L^2\Omega^{2,0}(\reg(X),h)\oplus L^2\Omega^{2,2,}(\reg(X),h)\rightarrow L^2\Omega^{2,1}(\reg(X),h)
\end{equation}
 whose domain is $\mathcal{D}(\overline{\partial}_{2,0,\max})\oplus \mathcal{D}(\overline{\partial}_{2,1,\min}^t)\subset L^2\Omega^{2,0}(\reg(X),h)\oplus L^2\Omega^{2,2,}(\reg(X),h)$. Its adjoint is 
\begin{equation}
\label{rege}
\overline{\partial}_{2,1,\max}+\overline{\partial}^t_{2,0,\min}:L^2\Omega^{2,1}(\reg(X),h)\rightarrow L^2\Omega^{2,0}(\reg(X),h) \oplus L^2\Omega^{2,2,}(\reg(X),h)
\end{equation}
with domain given by   $\mathcal{D}(\overline{\partial}_{2,1,\max})\cap \mathcal{D}(\overline{\partial}_{2,0,\min}^t)\subset L^2\Omega^{2,1}(\reg(X),h)$. In order  to state the next results we also need to recall briefly the existence of a resolution  of singularities. This is a deep topic in complex analytic geometry and we refer to the celebrated work of Hironaka \cite{Hiro},  to \cite{BiMi} and to \cite{HH} for a thorough discussion  on this subject. Furthermore we refer to \cite{GMMI} and to \cite{MaMa}  for a quick introduction. Below we simply recall what is strictly necessary for our purposes.\\ Let $X$ be a compact irreducible complex space. Then there exists a compact complex manifold $M$, a divisor with only normal crossings $D\subset M$ and a surjective holomorphic map $\pi:M\rightarrow X$ such that $\pi^{-1}(\sing(X))=D$ and 
\begin{equation}
\label{hiro}
\pi|_{M\setminus D}: M\setminus D\longrightarrow X\setminus \sing(X)
\end{equation}
is a biholomorphism. We have now all the ingredients to state the next corollary.

\begin{cor}
\label{kean}
In the setting of Th. \ref{coccabelladezio}. Let us label by $(\overline{\partial}_{2,\max}+\overline{\partial}_{2,\min}^t)^+$ the operator defined in \eqref{isabella}. Then $(\overline{\partial}_{2,\max}+\overline{\partial}_{2,\min}^t)^+$ is a Fredholm operator on its domain endowed with the graph norm and its index satisfies
\begin{equation}
\label{jilm}
\ind((\overline{\partial}_{2,\max}+\overline{\partial}_{2,\min}^t)^+)=\sum_{q=0}^2(-1)^q\Tr(e^{-t\Delta_{\overline{\partial},2,q,\abs}}).
\end{equation}
In particular we have 
\begin{equation}
\label{jilmx}
\chi(M,\mathcal{K}_{M})=\sum_{q=0}^2(-1)^q\Tr(e^{-t\Delta_{\overline{\partial},2,q,\abs}})
\end{equation}
where $\pi:M\rightarrow X$ is any resolution of $X$, $\mathcal{K}_{M}$ is the sheaf of holomorphic $(2,0)$-forms on $M$ and $\chi(M,\mathcal{K}_{M})=\sum_{q=0}^2(-1)^q\dim(H^q(M,\mathcal{K}_{M}))$.
\end{cor} 

\begin{proof}
That $(\overline{\partial}_{2,\max}+\overline{\partial}_{2,\min}^t)^+$ is a Fredholm operator  is clear from Th. \ref{lillottina}. The equality \eqref{jilm} follows by the fact that  $$\ind((\overline{\partial}_{2,\max}+\overline{\partial}_{2,\min}^t)^+)=\sum_{q=0}^2(-1)^q\ker(\Delta_{\overline{\partial},2,q,\abs})=\sum_{q=0}^2(-1)^q\Tr(e^{-t\Delta_{\overline{\partial},2,q,\abs}})$$ see for instance the proof of Cor. 5.3 in \cite{FBei} for the details.  The equality \eqref{jilmx} follows by \eqref{jilm} and the results established in \cite{JRu}.
\end{proof}

\begin{teo}
\label{coccabelladezios}
In the  setting of Th. \ref{coccabelladezio}. Let $q\in \{0,1,2\}$ and consider the operator
\begin{equation}
\label{resileq}
\Delta_{\overline{\partial},2,q,}^{\mathcal{F}}:L^2\Omega^{2,q}(\reg(X),h)\rightarrow L^2\Omega^{2,q}(\reg(X),h)
\end{equation}
that is the Friedrichs extension of $\Delta_{\overline{\partial},2,q}:\Omega^{2,q}_c(\reg(X))\rightarrow \Omega^{2,q}_c(\reg(X))$.
Let  $$0\leq \mu_1\leq \mu_2\leq...\leq \mu_k\leq...$$ be the eigenvalues of \eqref{resileq} and 
let $$0\leq \lambda_1\leq \lambda_2\leq...\leq \lambda_k\leq...$$ be the eigenvalues of \eqref{resile}. Then we have the following  inequality for every $k\in \mathbb{N}$
\begin{equation}
\label{compare}
\lambda_{k}\leq \mu_k.
\end{equation}
In particular we have 
\begin{equation}
\label{resileh}
\lim \inf \mu_k k^{-\frac{1}{2}}>0
\end{equation}
as $k\rightarrow \infty$.\\ Consider now the heat operator associated to \eqref{resileq}
\begin{equation}
\label{resiley}
e^{-t\Delta_{\overline{\partial},2,q}^{\mathcal{F}}}:L^2\Omega^{2,q}(\reg(X),h)\rightarrow L^2\Omega^{2,q}(\reg(X),h).
\end{equation}
Then \eqref{resiley} is a trace class operator and 
\begin{equation}
\label{pololo}
\Tr(e^{-t\Delta_{\overline{\partial},2,q}^{\mathcal{F}}})\leq \Tr(e^{-t\Delta_{\overline{\partial},2,q,\abs}}).
\end{equation}
In particular we have the following estimate for  $\Tr(e^{-t\Delta_{\overline{\partial},2,q}^{\mathcal{F}}})$
\begin{equation}
\label{resilezk}
\Tr(e^{-t\Delta_{\overline{\partial},2,q}^{\mathcal{F}}})\leq  B_qt^{-2}
\end{equation}
for $t\in (0,1]$ and some constant $B_q>0$.
\end{teo}

\begin{proof}
The inequality \eqref{compare} follows by the min-max Theorem arguing as in the proof of Th. 5.5 in \cite{FBei}. The remaining properties follow now immediately  using \eqref{compare} and Th. \ref{coccabelladezio}.
\end{proof}

We have the following application  concerning  forms of bi-degree $(1,0)$.

\begin{teo}
\label{supercoccabelladezio}
In the  setting of Th. \ref{coccabelladezio}. Assume moreover that $h$ is K\"ahler. Consider the operator
\begin{equation}
\label{sacripante}
\Delta_{\overline{\partial},1,0}^{\mathcal{F}}:L^2\Omega^{1,0}(\reg(X),h)\rightarrow L^2\Omega^{1,0}(\reg(X),h)
\end{equation}
that is the Friedrichs extension of $\Delta_{\overline{\partial},1,0}:\Omega^{1,0}_c(\reg(X))\rightarrow \Omega^{1,0}_c(\reg(X))$.
Then \eqref{sacripante} has discrete spectrum. Let  $$0\leq \mu_1\leq \mu_2\leq...\leq \mu_k\leq...$$ be the eigenvalues of \eqref{sacripante}. We have the following asymptotic  inequality 
\begin{equation}
\label{zucca}
\lim \inf \mu_k k^{-\frac{1}{2}}>0
\end{equation}
as $k\rightarrow \infty$.\\ Finally consider  the heat operator associated to \eqref{sacripante}
\begin{equation}
\label{gargantuesco}
e^{-t\Delta_{\overline{\partial},1,0}^{\mathcal{F}}}:L^2\Omega^{1,0}(\reg(X),h)\rightarrow L^2\Omega^{1,0}(\reg(X),h).
\end{equation}
Then \eqref{gargantuesco} is a trace class operator and its trace satisfies the following estimate
\begin{equation}
\label{pantagruelico}
\Tr(e^{-t\Delta_{\overline{\partial},1,0}^{\mathcal{F}}})\leq  Ct^{-2}
\end{equation}
for $t\in (0,1]$ and some constant  $C>0$.
\end{teo}

\begin{proof}
Using \eqref{cicci} and the Hodge star operator we have $*\circ \Delta_{\overline{\partial},2,1}=\Delta_{\partial,1,0}\circ *$ on $\Omega^{2,1}_c(\reg(X))$. It is easy to check that the previous equality implies that $*(\mathcal{D}(\Delta_{\overline{\partial},2,1}^{\mathcal{F}}))=\mathcal{D}(\Delta_{\partial,1,0}^{\mathcal{F}})$ and that  $*\circ \Delta_{\overline{\partial},2,1}^{\mathcal{F}}=\Delta_{\partial,1,0}^{\mathcal{F}}\circ *$ on $\mathcal{D}(\Delta_{\overline{\partial},2,1}^{\mathcal{F}})$. Moreover, by the K\"ahler identities, we have $\Delta_{\partial,1,0}=\Delta_{\overline{\partial},1,0}$ on $\Omega^{1,0}_c(\reg(X))$ and therefore $\Delta_{\partial,1,0}^{\mathcal{F}}=\Delta_{\overline{\partial},1,0}^{\mathcal{F}}$ on $L^2\Omega^{1,0}(\reg(X),h)$ as unbounded self-adjoint operators. Summarizing  we have shown that $*(\mathcal{D}(\Delta_{\overline{\partial},2,1}^{\mathcal{F}}))=\mathcal{D}(\Delta_{\overline{\partial},1,0}^{\mathcal{F}})$ and that $*\circ \Delta_{\overline{\partial},2,1}^{\mathcal{F}}=\Delta_{\overline{\partial},1,0}^{\mathcal{F}}\circ *$ on  $\mathcal{D}(\Delta_{\overline{\partial},2,1}^{\mathcal{F}})$. Now all the statements of this theorem follows by Th. \ref{coccabelladezios}.
\end{proof}

In the last part of this section we gather various corollaries that arise, through  \eqref{cicuta} and Prop. \ref{occhiodibue}, as immediate consequences of the results proved so far. 
\begin{cor}
\label{sigari}
In the  setting of Th. \ref{lillottina}. For each $q=0,1,2$ the operator
\begin{equation}
\label{salicornia}
\Delta_{\overline{\partial},0,q,\rel}:L^2\Omega^{0,q}(\reg(X),h)\rightarrow L^2\Omega^{0,q}(\reg(X),h)
\end{equation}
has discrete spectrum. Let  $$0\leq \lambda_1\leq \lambda_2\leq...\leq \lambda_k\leq...$$ be the eigenvalues of \eqref{salicornia}. Then we have the following asymptotic  inequality 
\begin{equation}
\label{drive}
\lim \inf \lambda_k k^{-\frac{1}{2}}>0
\end{equation}
as $k\rightarrow \infty$.\\ Finally consider  the heat operator associated to \eqref{salicornia}
\begin{equation}
\label{ottenebrare}
e^{-t\Delta_{\overline{\partial},0,q,\rel}}:L^2\Omega^{0,q}(\reg(X),h)\rightarrow L^2\Omega^{0,q}(\reg(X),h).
\end{equation}
Then \eqref{ottenebrare} is a trace class operator and  we have the following estimate for its trace
\begin{equation}
\label{obnubilare}
\Tr(e^{-t\Delta_{\overline{\partial},0,q,\rel}})\leq  C_qt^{-2}
\end{equation}
for $t\in (0,1]$ and some constant $C_q>0$.
\end{cor}

\begin{proof}
Using \eqref{cicuta} and Prop. \ref{occhiodibue} we have that any form $\omega\in L^2\Omega^{2,q}(\reg(X),h)$ lies in $\mathcal{D}(\Delta_{\overline{\partial},2,q,\abs})$ if and only if $c(*\omega)\in \mathcal{D}(\Delta_{\overline{\partial},0,2-q,\rel})$ and if this is the case then we have $c(*(\Delta_{\overline{\partial},2,q,\abs}\omega))=\Delta_{\overline{\partial},0,2-q,\rel}(c(*\omega))$. Since  $c\circ *:L^2\Omega^{2,q}(\reg(X),h)\rightarrow L^2\Omega^{0,2-q}(\reg(X),g)$ is a continuous and  bijective $\mathbb{C}$-antilinear isomorphism with continuous inverse the conclusion  follows now by Th. \ref{lillottina} and Th. \ref{coccabelladezio}.
\end{proof}

\begin{cor}
\label{supercoccabelladeziow}
In the  setting of Th. \ref{lillottina}. For each $q=0,1,2$ the operator
\begin{equation}
\label{sacripantez}
\Delta_{\overline{\partial},0,q}^{\mathcal{F}}:L^2\Omega^{0,q}(\reg(X),h)\rightarrow L^2\Omega^{0,q}(\reg(X),h)
\end{equation}
has discrete spectrum. Let  $$0\leq \mu_1\leq \mu_2\leq...\leq \mu_k\leq...$$ be the eigenvalues of \eqref{sacripantez}. Then we have the following inequality $$\mu_k\geq \lambda_k$$ where $0\leq \lambda_1\leq...\leq\lambda_k\leq...$ are the eigenvalues of \eqref{salicornia}. Moreover we have the following asymptotic  inequality 
\begin{equation}
\label{zuccaz}
\lim \inf \mu_k k^{-\frac{1}{2}}>0
\end{equation}
as $k\rightarrow \infty$.\\ Consider the heat operator associated to \eqref{sacripantez}
\begin{equation}
\label{gargantuescoz}
e^{-t\Delta_{\overline{\partial},0,q}^{\mathcal{F}}}:L^2\Omega^{0,q}(\reg(X),h)\rightarrow L^2\Omega^{0,q}(\reg(X),h).
\end{equation}
Then \eqref{gargantuescoz} is a trace class operator.  We have the following inequality $$\Tr(e^{-t\Delta_{\overline{\partial},0,q}^{\mathcal{F}}})\leq \Tr(e^{-t\Delta_{\overline{\partial},0,q,\rel}})$$ for every $t>0$ and 
furthermore   $\Tr(e^{-t\Delta_{\overline{\partial},0,q}^{\mathcal{F}}})$ satisfies the following estimate
\begin{equation}
\label{pantagruelicoz}
\Tr(e^{-t\Delta_{\overline{\partial},0,q}^{\mathcal{F}}})\leq  C_qt^{-2}
\end{equation}
for $t\in (0,1]$ and some constant $C_q>0$.
\end{cor}

\begin{proof}
This corollary follows by Th. \ref{lillottina} and  Th. \ref{coccabelladezios} using \eqref{cicuta} and Prop. \ref{occhiodibue} as in the  proof of Cor. \ref{sigari}.
\end{proof}

\begin{cor}
\label{supercoccabelladeziof}
In the setting of Th. \ref{supercoccabelladezio}. The operator
\begin{equation}
\label{sacripantef}
\Delta_{\overline{\partial},1,2}^{\mathcal{F}}:L^2\Omega^{1,2}(\reg(X),h)\rightarrow L^2\Omega^{1,2}(\reg(X),h)
\end{equation}
 has discrete spectrum. Let  $$0\leq \mu_1\leq \mu_2\leq...\leq \mu_k\leq...$$ be the eigenvalues of \eqref{sacripantef}. Then we have the following asymptotic  inequality 
\begin{equation}
\label{zuccaf}
\lim \inf \mu_k k^{-\frac{1}{2}}>0
\end{equation}
as $k\rightarrow \infty$.\\ Finally consider  the heat operator associated to \eqref{sacripantef}
\begin{equation}
\label{gargantuescof}
e^{-t\Delta_{\overline{\partial},1,2}^{\mathcal{F}}}:L^2\Omega^{1,2}(\reg(X),h)\rightarrow L^2\Omega^{1,2}(\reg(X),h).
\end{equation}
Then \eqref{gargantuescof} is a trace class operator and its trace satisfies the following estimate
\begin{equation}
\label{pantagruelicof}
\Tr(e^{-t\Delta_{\overline{\partial},1,2}^{\mathcal{F}}})\leq  Ct^{-2}
\end{equation}
for $t\in (0,1]$ and some constant $C>0$.
\end{cor}

\begin{proof}
This corollary follows by Th. \ref{supercoccabelladezio} using \eqref{cicuta} and Prop. \ref{occhiodibue} as in the  proof of Cor. \ref{sigari}.
\end{proof}

We conclude the paper with the following  McKean-Singer formula concerning   $(L^2\Omega^{0,q}(\reg(X),h),\overline{\partial}_{0,q,\min})$. Let $X$ and $h$ be as in Th. \ref{lillottina}. Consider the operator
\begin{equation}
\label{isabellaf}
\overline{\partial}_{\min}+\overline{\partial}^t_{0,1,\max}:L^2(\reg(X),h)\oplus L^2\Omega^{0,2}(\reg(X),h)\rightarrow L^2\Omega^{0,1}(\reg(X),h)
\end{equation}
 whose domain is $\mathcal{D}(\overline{\partial}_{\min})\oplus \mathcal{D}(\overline{\partial}_{0,1,\max}^t)\subset L^2(\reg(X),h)\oplus L^2\Omega^{0,2}(\reg(X),h)$. Its adjoint is 
\begin{equation}
\label{regef}
\overline{\partial}_{0,1,\min}+\overline{\partial}^t_{\max}:L^2\Omega^{0,1}(\reg(X),h)\rightarrow L^2(\reg(X),h) \oplus L^2\Omega^{0,2}(\reg(X),h)
\end{equation}
with domain given by   $\mathcal{D}(\overline{\partial}_{0,1,\min})\cap \mathcal{D}(\overline{\partial}_{\max}^t)\subset L^2\Omega^{0,1}(\reg(X),h)$.

\begin{cor}
\label{mckeanwx}
In the setting of Th. \ref{coccabelladezio}. Let us label by $(\overline{\partial}_{0,\min}+\overline{\partial}_{0,\max}^t)^+$ the operator defined in \eqref{isabellaf}. Then $(\overline{\partial}_{0,\min}+\overline{\partial}_{0,\max}^t)^+$ is a Fredholm operator on its domain endowed with the graph norm and 
\begin{equation}
\label{jilmqz}
\ind((\overline{\partial}_{0,\min}+\overline{\partial}_{0,\max}^t)^+)=\sum_{q=0}^2(-1)^q\Tr(e^{-t\Delta_{\overline{\partial},0,q,\rel}}).
\end{equation}
In particular we have 
\begin{equation}
\label{jilmxqz}
\chi(M,\mathcal{O}_{M})=\sum_{q=0}^2(-1)^q\Tr(e^{-t\Delta_{\overline{\partial},0,q,\rel}})
\end{equation}
where $\pi:M\rightarrow X$ is any resolution of $X$ and $\chi(M,\mathcal{O}_{M})=\sum_{q=0}^2(-1)^q\dim(H^{0,q}_{\overline{\partial}}(M))$. 
\end{cor} 

\begin{proof}
The equality \eqref{jilmqz} can be proved arguing as in the proof of  \eqref{jilm}. The equality \eqref{jilmxqz} follows by \eqref{jilmqz} and the results established in \cite{JRu}.
\end{proof}

\end{document}